\documentclass[11pt]{amsart}
\usepackage{graphicx}
\usepackage{hyperref}
\usepackage{amssymb, amsmath, amscd}
\usepackage{bm}
\usepackage{float}
\usepackage{multirow}
\usepackage[all,cmtip]{xy}

\newtheorem{theorem}{Theorem}[section]
\newtheorem{corollary}[theorem]{Corollary}
\newtheorem{lemma}[theorem]{Lemma}
\newtheorem{prop}[theorem]{Proposition}

\newtheorem{defi}[theorem]{Definition}
\newtheorem{remark}[theorem]{Remark}

\theoremstyle{definition}

\newcommand{\mb}{\mathbf}
\newcommand{\al}{\alpha}
\newcommand{\Z}{\mathbb{Z}}
\newcommand{\paren}[1]{\left( #1 \right)}
\newcommand{\veco}[1]{\overrightarrow{#1}}
\newcommand{\raw}{\rightarrow}

\begin{document}

\title[Combinatorial Sutured TQFT]{Combinatorial Sutured TQFT as Exterior Algebra}

\author[Evan Fink]{Evan Fink}

\address{Department of Mathematics,
University of Georgia, 
Athens, GA 30602}
\email{efink@math.uga.edu}

\begin{abstract}
The idea of a \emph{sutured topological quantum field theory} was introduced by Honda, Kazez and Mati\'{c} in \cite{HKM}. A sutured TQFT associates a group to each \emph{sutured surface} and an element of this group to each \emph{dividing set} on this surface.  The notion was originally introduced to talk about contact invariants in Sutured Floer Homology.  We provide an elementary example of a sutured TQFT, which comes from taking exterior algebras of certain singular homology groups.  We show that this sutured TQFT coincides with that of \cite{HKM} using $\Z_2$-coefficients. The groups in our theory, being exterior algebras, naturally come with the structure of a ring with unit.  We give an application of this ring structure to understanding tight contact structures on solid tori.
\end{abstract}

\maketitle

\section{Introduction}

This paper will concern algebraic objects -- \emph{sutured topological quantum field theories} -- associated to a type of decorated surface -- \emph{sutured surfaces} -- and certain embedded multicurves -- \emph{dividing sets} -- on these surfaces.  While we find the objects of study interesting from an algebraic perspective, everything herein is motivated by contact geometry.

A \emph{contact structure} $\xi$ on a 3-manifold $M$ is a nowhere-integrable plane field. To use cut-and-paste methods to study contact structures, one uses the notion of a \emph{convex surface}, introduced by Giroux \cite{Giroux}: this is a surface $\Sigma$ in $(M, \xi)$ such that there is a vector field $X$ that is everywhere transverse to $\Sigma$ and whose flow preserves $\xi$.  A convex surface is typically taken to be closed, or with boundary that is Legendrian (i.e., tangent to $\xi$).  The \emph{dividing set} of a convex surface is the multicurve $\Gamma \subset \Sigma$ of points $p$ where $X_p \in \xi_p$.  Of course, as defined, the dividing set on $\Sigma$ depends also on the particular choice of vector field $X$; but in fact it is shown in \cite{Giroux} that the restriction of $\xi$ to a neighborhood of $\Sigma$ is determined (up to isotopy through contact structures) by the isotopy class of its dividing set.

A sutured manifold $(M, \Gamma)$ is a compact, oriented, not necessarily
connected 3-manifold $M$ with boundary, together with an oriented embedded 1-manifold $\Gamma \subset \partial M$, such that $\partial M \setminus \Gamma$ is divided into two disjoint subsurfaces $R^+$ and $R^-$, such that $\Gamma$ is the oriented boundary of $\overline{R^+}$ and $-\Gamma$ is the oriented boundary of $\overline{R^-}$.  (The original definition, due to Gabai \cite{Gabai}, is slightly different.)    Sutured manifolds are important in three-dimensional contact geometry because for a contact 3-manifold with convex boundary, the dividing set on the boundary forms a set of sutures. 

In \cite{juhasz}, Juh\'{a}sz defines an invariant of a sutured manifold $(M, \Gamma)$, the \emph{Sutured Floer homology} $SFH(M, \Gamma)$.   (Technically, the sutured manifold must be \emph{balanced}: every boundary component must contain a suture, and $R^+$ and $R^-$ must have equal Euler characterisitics.) This invariant, a variation on the Heegaard Floer homology of Ozsv\'{a}th and Szab\'{o} \cite{ozsz}, is defined by counting holomorphic curves in spaces associated to Heegaard splittings of $M$ that are suitably adapted to $\Gamma$.  In \cite{HKM1}, Honda, Kazez and Mati\'{c} define an invariant $EH(\xi) \in SFH(-M, -\Gamma)$ of a contact three-manifold with convex boundary $(M, \xi)$, where $\xi$ is a postive and cooriented contact structure, and $\Gamma$ is the dividing set on the boundary.

The invariant $EH(\xi)$ is studied further in \cite{HKM}, where interesting observations are made about $S^1$-invariant contact structures on manifolds of the form $S^1 \times \Sigma$, where $\Sigma$ is a surface with non-empty boundary, $S^1 \times \partial \Sigma$ is convex, and $S^1$ acts on the manifold and the contact planes in the obvious manner.   For these manifolds, the dividing set on $\partial(S^1 \times \Sigma)$ will be $\Gamma = S^1 \times F$ where $F$ is a finite set of points in $\partial \Sigma$; and the contact structure will in fact be determined by looking at the dividing set on the convex surface $\{pt\} \times \Sigma$.
It is shown in \cite{HKM} that the relationship between dividing sets on $\Sigma$ and the contact invariants of the associated $S^1$-invariant contact structure can be described in a manner that is somewhat reminiscent of a topological quantum field theory.  They proceed to define a new notion, a \emph{sutured TQFT}.  

\subsection{Summary of Sutured TQFT}

To summarize the idea of Sutured TQFT and our results, let us make some imprecise
definitions, reviewing the precise definitions in Sections \ref{sec:not} and \ref{sec:tqft}. 

Informally, a \emph{sutured surface} is a pair $(\Sigma, F)$, where $\Sigma$ is a surface without closed components and $F$ denotes a finite subset of $\partial \Sigma$ (the \emph{sutures}). Each component of $\partial \Sigma$ contains a positive even number of points of $F$, which divide $\partial \Sigma$ into alternating \emph{positive} and \emph{negative} arcs. 

Likewise, we can informally define a \emph{dividing set} $K$ on $(\Sigma, F)$ to be an oriented 1-manifold in a sutured surface, with boundary in $F$, which divides the surface into disjoint \emph{positive} and \emph{negative} regions such that no two neighboring regions bare the same sign.  

Finally, a \emph{gluing} on $(\Sigma, F)$ is, roughly, a quotient map $\phi_{\tau}$ from $(\Sigma, F)$ to a new sutured surface $(\Sigma_{\tau}, F_{\tau})$, obtained by identifying submanifolds of $\partial \Sigma$ by a diffeomorphism $\tau$, in a manner which appropriately respects the sutures. The image of a dividing set $K$ on $(\Sigma, F)$ will be a dividing set $K_{\tau}$ on $(\Sigma_{\tau}, F_{\tau})$.

\begin{defi}
\label{defi:quote}
A \emph{sutured TQFT} $\mathcal{F}$ is an assignment of

\begin{itemize}

\item a graded abelian group $V(\Sigma, F)^{\mathcal{F}}$ to each sutured surface $(\Sigma, F)$ (the \emph{sutured surface group}); 

\item a subset $c^{\mathcal{F}}(K)$ of $V(\Sigma, F)^{\mathcal{F}}$ of the form $\{x, -x\}$ to each dividing set $K$ (the \emph{contact subset});

\item and a pair of maps $\pm\Phi_{\tau}^{\mathcal{F}}: V(\Sigma, F)^{\mathcal{F}} \rightarrow V(\Sigma_{\tau}, F_{\tau})^{\mathcal{F}}$ to each gluing producing $(\Sigma_{\tau}, F_{\tau})$ from $(\Sigma, F)$ (the \emph{gluing morphism set}).

\end{itemize}

Among other things, we require that $c^{\mathcal{F}}(K) = \{0\}$ if $K$ contains a homotopically trivial closed component; and that $\pm\Phi^{\mathcal{F}}_{\tau}$ takes $c^{\mathcal{F}}(K)$ to $c^{\mathcal{F}}(K_{\tau})$.  
\end{defi}

In \cite{HKM}, Sutured Floer homology is used to define a sutured TQFT.  The somewhat strange sign issues noted there explain the unsightly use of contact subsets rather than elements -- in their theory, there is no coherent way to pick out distinguished elements from these subsets. 

\subsection{Main Results}
The goal of this paper is to define an alternative construction satisfying the sutured TQFT axioms.  Our construction replaces Sutured Floer homology groups with exterior algebras on certain singular homology groups.  We make assignments as above, which we label as $V(\Sigma, F)^X$, $c^X(K),$ and $\pm\Phi^X_{\tau}$; the relevant Definitions are \ref{defi:sutalg}, \ref{defi:contelt}, and \ref{defi:gluemorph}, respectively.  

Our main results, stated more precisely as Theorems \ref{thm:mainthm} and \ref{thm:isom}, are the following.  

\begin{theorem}
The assignments $V(\Sigma, F)^X$, $c^X(K)$ and $\pm\Phi^X_{\tau}$ form a sutured TQFT.  
\end{theorem}

We denote this TQFT by $\mathcal{X}$, or $\mathcal{X}(\Z_2)$ when taken using $\Z_2$ coefficients.  

\begin{theorem}
Let $\mathcal{SFH}(\Z_2)$ denote the sutured TQFT of [4], taken with $\Z_2$ coefficients; then $\mathcal{SFH}(\Z_2)$ is equivalent to $\mathcal{X}(\Z_2)$.
\end{theorem}  

The notion of equivalence is detailed in Definition \ref{defi:equiv}.  We prove this by showing that in this sense, all such TQFTs over $\Z_2$ are equivalent,provided they obey a couple of extra conditions.  We expect that the same is true over $\Z$ (indeed, we have a sketch of a proof), but showing this requires a more involved argument, due to the aforementioned sign issues.

At the moment, our view is that this construction accomplishes three things.  First, it provides an easily computed tool for understanding contact elements in Sutured Floer Homology.  Secondly, it elucidates the structure of Suture Floer Homology itself: since they are exterior algebras, the groups $V(\Sigma, F)^X$ of our TQFT all have the structure of a ring with unit.  This makes some sense in light of the fact that the Sutured Floer homology of a sutured manifold $(M, \Gamma)$ is known to admit an action by $\Lambda\paren{H_1(M)}$ \cite{Ni}.  Finally, it provides a tool for detecting when a convex sphere or disk with a given dividing set has a tight neighborhood.  Of course, we have a straightfoward test for this (Giroux's criterion \cite{Giroux}, see Section \ref{sec:disks}); but, crucially, the TQFT behaves nicely with respect to cutting and pasting.  We discuss this last point in Section \ref{sec:disks}, giving an application to the study of tight contact structures on solid tori; this utilizes the exterior product, showing that this operation has some significance at least for the case of disks.

\subsection{Further questions}

We have several questions that we would like to investigate.  The first one is, of course, to show that there is an isomorphism between our TQFT and the SFH one using $\Z$ coefficients.  Two avenues to investigate this would be to look carefully at the  $\Lambda\paren{H_1(M)}$-action on Sutured Floer homology, or to strengthen the approach we use herein for the $\Z_2$ case, showing that there is little freedom in how an TQFT assigns contact elements to dividing sets.  We believe that we can prove this using the latter strategy.

We also wonder if there is a similar construction that mimics Sutured Floer homology with twisted coefficients.  Note that the details of such a construction would have to be somewhat different.  Indeed, our construction, almost by definition, assigns $0$ to any dividing set which is \emph{isolating} (see Section \ref{sec:not} for the definition); in \cite{Massot}, Massot notes that the twisted coefficient analogue for Sutured Floer homology assigns non-trivial invariants to some isolating dividing sets.  

We would like to know to what exent our TQFT can be helpfully extended to study contact structures on non-trivial surface bundles over $S^1$.  In addition, we would like to know the extent to which the TQFT illuminates the \emph{contact category}, described in \cite{contcat}.  

The considerations of Section \ref{sec:disks} can certainly be extended to examine the tight contact structures on handlebodies with a given convex boundary.  We wonder exactly how much can be said about this.

We find the special case of sutured disks to be a fertile ground to study the algebraic structure of this theory.  The central question we have about these is, ``to what extent can we \emph{algebraically} characterize the elements which appear as contact elements?''  Questions like this one have been studied extensively in a series of papers by Daniel Mathews \cite{Mathews}, \cite{Mathews2}, \cite{Mathews3}, \cite{itsybitsy}.  Finally, we should mention that these papers also contain a wealth of fascinating ideas relating the sutured TQFT notion to more standard ideas from quantum field theory (among other things), which we would hope to be able to illumnate.

\subsection{Organization}

In Section \ref{sec:not}, we set notation.  In Section \ref{sec:assign}, we define the assignments of our Exterior TQFT: sutured surface algebra, contact subsets/elements, and gluing morphisms.  In Section \ref{sec:tqft}, we review the notion of a sutured TQFT from \cite{HKM}, and establish that the Exterior TQFT does fit this notion; we then compare properties of this TQFT with those of the Sutured Floer Homology TQFT, before establishing the equivalence of the two TQFTs over $\Z_2$.  In Section \ref{sec:disks}, we show how the product operation on our algebra has some significance, related to Giroux's criterion for convex spheres with tight neighborhoods, and we describe an application to understanding contact structures on solid tori.

\subsection{Acknowledgements}

The author thanks Eric Burgess, Will Kazez, Daniel Mathews and Gordana Mati\'{c} for valuable insights and comments.

\section{Sutured Surfaces, Dividing Sets, and Gluing}
\label{sec:not}

We recall and enhancing some familiar notions from contact geometry, proving some basic results.

\subsection{Sutured surfaces}
\label{sec:sutsur}
Let $\Sigma$ be a compact, oriented surface, possibly disconnected, with no closed components.  As mentioned, a \emph{sutured surface} is essentially just a surface together with some points on the boundary, which dividing the boundary into alternating positive and negative arcs.  However, to facilitate our exposition, we cram some extra data into the notion up front.  So, the definition we use is the following.

\begin{defi}
\label{def:sutsur}
A \emph{sutured surface} is the data 
$$(\Sigma, F^+, F^-, \al^+, \al^-)$$ described as follows.

\begin{itemize}

\item $\Sigma$ is a compact, oriented, possibly disconnected surface with no closed components.

\item $F^+, F^-, \al^+,$ and $\al^-$ are each (pairwise disjoint) finite subsets of $\partial \Sigma$ which satisfy the following properties: if $C$ is a component of $\partial \Sigma$, then the intersection of $C$ with each of these sets must have the same cardinality, which must be positive; if we start on a point of $F^+$, and traverse $C$ in the direction given by its orientation, we encounter points of these sets in the order $F^+, \al^+, F^-, \al^-, F^+, \al^+, \ldots$. 

\end{itemize}
\end{defi}
 
We will generally denote such an object simply by $(\Sigma, \mb{F})$, with $\mb{F}$ denoting all the other data.  The points in $F = F^+ \cup F^-$ are the \emph{sutures}, and the points of $\al = \al^+ \cup \al^-$ are the \emph{vertices}.  Let $n(\mb{F})$ be the common cardinality of $F^+, F^-, \al^+$ and $\al^-$.  Also, let 
$$L(\Sigma, \mb{F}) = n(\mb{F}) -\chi(\Sigma).$$

\begin{lemma}
\label{lem:lrank}
The rank of $H_1(\Sigma, \al^+)$ is $L(\Sigma, \mb{F})$.
\end{lemma}

\begin{proof}
Let $r$ denote the rank of $H_1(\Sigma, \al^+)$.
Consider a portion of the long exact sequence of the pair $(\Sigma, \al^+)$,
$$ H_1(\al^+) \rightarrow H_1(\Sigma) \rightarrow H_1(\Sigma, \al^+) \rightarrow H_0(\al^+) \rightarrow H_0(\Sigma) \rightarrow H_0(\Sigma, \al^+).$$
The first group is trivial; so is the last, since we have demanded that $\Sigma$ have no closed components, and that all components of $\partial \Sigma$ contain vertices. The ranks of $H_1(\Sigma)$, $H_0(\al^+)$ and $H_0(\Sigma)$ are respectively $b_1(\Sigma)$, $n(\mb{F})$, and $b_0(\Sigma)$; therefore, $b_1(\Sigma) + n(\mb{F}) = r + b_0(\Sigma)$. Since $\Sigma$ has no closed components, $r$  equals $n(\mb{F}) - \chi(\Sigma)$.
\end{proof}

Every component of $\partial \Sigma \setminus F$ contains exactly one point in $\al^+ \cup \al^-$.  Let $A^+$ denote the union of the closures of those components which contain a point of $\al^+$, and refer to these components as \emph{positive boundary arcs}.  Similarly define $A^-$, whose components are \emph{negative boundary arcs}.  Both $A^+$ and $A^-$ are understood to have orientation inherited from $\partial \Sigma$.

The following observations will inform our notation throughout.  
First, notice that $(\Sigma, \al^+)$ is homotopy equivalent to $(\Sigma, A^+)$, and hence $H_1(\Sigma, \al^+)$ is naturally identified with $H_1(\Sigma, A^+)$ (and likewise for $H_1(\Sigma, \al^-)$ and $H_1(\Sigma, A^-)$).  Also, by Poincar\'{e} duality, $H^1(\Sigma, A^+) \cong H_1(\Sigma, A^-)$.  Hence, we may think of $H_1(\Sigma, \al^-)$ as $H_1(\Sigma, \al^+)^*$.  

More generally, we may trade out $A^+$ for $\al^+$ and $A^-$ for $\al^-$ in practically every homology group we mention throughout; we generally will do so without mention going forth.

\subsection{Dividing sets and homology orientations}

We next recall the notion of a \emph{dividing set}.

\begin{defi}
\label{defi:divset}
A \emph{dividing set} $K$ on a sutured surface $(\Sigma, \mb{F})$ is a properly embedded, oriented 1-dimensional submanifold with $\partial K = F$, such that:

\begin{itemize}

\item each component of $\Sigma \setminus K$ can be designated as \emph{positive} or \emph{negative}, with the two components adjacent to any component of $K$ having different signs; 

\item if $R^+(K)$ is the union of the closures of the positive components, then we require $\partial R^+(K) = A^+ \cup K$ as oriented manifolds; 

\item and with $R^-(K)$ similarly defined, we require $\partial R^-(K) = A^- \cup -K$.

\end{itemize}
\end{defi}

We recall a crucial property of certain dividing sets.

\begin{defi}
An \emph{isolated component} of a dividing set $K$ on $(\Sigma, \mb{F})$ is a component of $\Sigma \setminus K$ that has trivial intersection with $\partial \Sigma$; and $K$ is \emph{non-isolating} if it has no isolated components.
\end{defi}

We characterize this condition.  To state this, let 
$$L(K) = n(\mb{F}) - \chi\paren{R^+(K)};$$
also, let $I^+(K)$ be the number of isolated components of $R^+(K)$.

\begin{prop}
\label{prop:isocomps}
For a dividing set $K$ on $(\Sigma, \mb{F})$, the rank of $H_1(R^+(K), \al^+)$ is $L(K) + I^+(K)$.  In particular, $R^+(K)$ has no isolated components if and only if the rank of $H_1(R^+(K), \al^+)$ is $L(K)$.  Also, $R^-(K)$ has no isolated components if and only the inclusion map $H_1(R^+(K), \al^+) \rightarrow H_1(\Sigma, \al^+)$ is injective.   
\end{prop}

\begin{proof}
For the first claim, we examine part of the long exact sequence of the pair $(R^+(K), \al^+)$,
$$ H_1(\al^+) \rightarrow H_1\paren{R^+(K)}  \rightarrow H_1(R^+(K), \al^+)  \rightarrow $$
$$ H_0(\al^+) \rightarrow H_0\paren{R^+(K)}  \rightarrow H_0(R^+(K), \al^+) \rightarrow 0. $$
Let $r$ be the rank of $H_1(R^+(K), \al^+)$.  The first group $H_1(\al^+)$ is trivial; the rank of $H_0(R^+(K), \al^+)$ is $I^+(K)$.  Then, we have
$$ b_1\paren{R^+(K)} + n(\mb{F}) + I^+(K) = r + b_0\paren{R^+(K)}. $$
Since $R^+(K)$ has no closed components, $\chi\paren{R^+(K)} = b_0\paren{R^+(K)} - b_1\paren{R^+(K)}$.  Therefore, $r = n(\mb{F}) - \chi\paren{R^+(K)} + I^+(K)$. 

For the claim about negative isolated components, we examine part of the long exact sequence of the triple $(\Sigma, R^+(K), \al^+)$,
$$ H_2(\Sigma, \al^+) \rightarrow H_2\paren{\Sigma, R^+(K)}  \rightarrow H_1(R^+(K), \al^+) \rightarrow H_1(\Sigma, \al^+). $$
Since $\Sigma$ has no closed components, $H_2(\Sigma, \al^+)$ is trivial.  As to $H_2\paren{\Sigma, R^+(K)}$, we have
$$ H_2\paren{\Sigma, R^+(K)} \cong H_2(R^-(K), K) \cong H^0(R^-(K), \al^-);$$
the first isomorphism is excision, and the second is Poincar\'{e} duality.  Therefore, this term vanishes if and only if $R^-(K)$ has no isolated components; this is equivalent to injectivity of the map from $H_1(R^+(K), \al^+)$ to $H_1(\Sigma, \al^+)$.
\end{proof}

\subsection{Gluing}
\label{sec:glue}

We now consider the operation of gluing.  

\begin{defi}
\label{defi:gluing}
Let $(\Sigma, \mb{F})$ be a sutured surface. A \emph{gluing} on $(\Sigma, \mb{F})$ is an orientation-reversing map $\tau: \gamma \rightarrow \gamma'$ of two (possibly disconnected) submanifolds $\gamma, \gamma'$ of $\partial \Sigma$, such that:

\begin{itemize}

\item $\gamma$ and $\gamma'$ each have boundary  (if any) in $\al$;  

\item $\tau$ maps $\al^+ \cap \gamma$ bijectively to $\al^+\cap \gamma'$, $\al^- \cap \gamma$ bijectively to $\al^-\cap \gamma'$, $F^+ \cap \gamma$ bijectively to $F^-\cap \gamma'$, and $F^- \cap \gamma$ bijectively to $F^+\cap \gamma'$, all without fixed points; 

\item $\tau$ restricted to $\gamma \setminus \al$ is a diffeomorphism onto its image; and

\item the surface $\Sigma_{\tau} = \Sigma/\sim$ has no closed components, where $\sim$ identifies $x \in \gamma$ with $\tau(x) \in \gamma'$. 
\end{itemize}

\end{defi}

Given a gluing $\tau$ on a sutured surface $(\Sigma, \mb{F})$, let $\phi_{\tau}: \Sigma \rightarrow \Sigma_{\tau}$ denote the quotient map.  As a consequence of our definition, $\Sigma_{\tau}$ will inherit sutures $F^{\pm}_{\tau}$ and vertices $\al^{\pm}_{\tau}$ -- those images of points in $F^{\pm}$ and $\al^{\pm}$ that lie on $\partial \Sigma_{\tau}$ -- ordered as in Definition \ref{def:sutsur}.  

\begin{defi}
For a gluing $\tau$ on $(\Sigma, \mb{F})$, the \emph{glued surface of $\tau$} is the sutured surface $(\Sigma_{\tau}, \mb{F}_{\tau})$, with $\mb{F}_{\tau} = (F^+_{\tau}, F^-_{\tau}, \al^+_{\tau}, \al^-_{\tau})$.
\end{defi}

Note that the sutures on $\gamma$ and $\gamma'$ will be mapped into the interior of $\Sigma_{\tau}$, and so their images will \emph{not} be part of $F^{\pm}_{\tau}$.  Likewise, the images of the vertices in the interior of $\gamma$ and $\gamma'$ will not be vertices in $\al^{\pm}_{\tau}$.
 
Given a dividing set $K$ on $(\Sigma, \mb{F})$, let $K_{\tau}$ denote the image of $K$ under $\phi_{\tau}$.  It is easy to see that $K_{\tau}$ is a dividing set meeting the demands of Definition \ref{defi:divset}.  

\subsection{Examples}

Figure \ref{fig:disk} depicts a sutured disk.  We set some notation for these which we will use throughout.  Given a positive integer $N$, let $(D^2, \mb{F}(N))$ be a sutured disk with $n(\mb{F}(N)) = N$. Starting on a positive suture, we will typically number the points we encounter going counterclockwise by $F_1, \al_1, F_2, \al_2, \ldots, F_{2N}, \al_{2N}$.  The odd numbers correspond to positive sutures and vertices, while the even numbers correspond to negative ones.  Take as a basis of $H_1(D^2, \al^+)$ the arcs $\beta_1, \beta_3, \ldots, \beta_{2N-3}$, where $\partial \beta_i = \al_{i+2} - \al_i$ (i.e., a representative is an arc oriented from $\al_i$ to $\al_{i+2}$). 

As in Figure \ref{fig:disk}, we will mark all our positive vertices by X's, and all our negative vertices by O's.  

\begin{figure}[!htbp]
	\centering
		\includegraphics[width=0.36\textwidth]{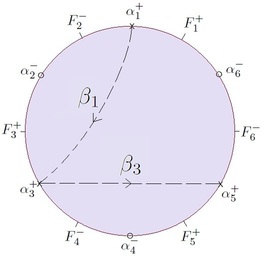}
	\caption{The sutured disk $\paren{D^2, \mb{F}(3)}$, with two arcs generating $H_1(D^2, \al^+)$.}
	\label{fig:disk}
\end{figure}

In Figure \ref{fig:cutopen}, we depict a gluing of a sutured disk.  The map $\tau$ identifies the two vertical parts of the boundary in the obvious manner, to produce a sutured annulus, while $\phi_{\tau}$ is the quotient map between surfaces.  Notice that the two top X's become one X in the glued surface, and likewise for the two bottom X's; meanwhile, the other three vertices on each side are ``swallowed'' into the interior of the glued surface.  In this figure, we also show a glued dividing set.   

\begin{figure}[!htbp]
	\centering
		\includegraphics[width=0.8\textwidth]{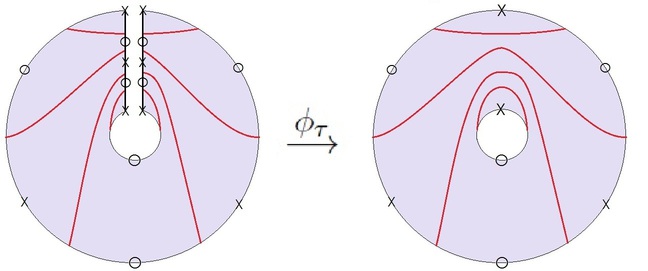}
	\caption{A gluing of sutured surfaces.}
	\label{fig:cutopen}
\end{figure}

\section{The Exterior TQFT Assignments}
\label{sec:assign}

Referring back to Definition \ref{defi:quote}, we now define our notions of sutured surface group, contact subset, and gluing morphism.  In fact, our group is an algebra.  In addition, we have reasonable ways to deal with the sign ambiguities of  contact elements and gluing morphisms.

After reviewing our algebraic setting and notation, this section just introduces the cast of characters: in the next section, we will precisely state what is expected of a sutured TQFT, and then show that our assignments form such a theory.

\subsection{Exterior algebras and notation}
\label{sec:extnot}

Let $M$ be a free $\Z$-module of rank $r$.  We will denote its exterior algebra by $\Lambda(M)$, with $\Lambda^i(M)$ the $i^{\mbox{\scriptsize{th}}}$ exterior grading. We will write $\Lambda^{top}(M)$.  Recall that $\Lambda(M_1 \oplus M_2)$ is canonically isomorphic to $\Lambda(M_1) \otimes \Lambda(M_2)$.

As usual, we have a pairing $\langle \cdot | \cdot \rangle: \Lambda\paren{M^*} \otimes \Lambda\paren{M} \rightarrow \Z$, defined on decomposable elements $f_1 \wedge \ldots \wedge f_k \in \Lambda^k(M^*)$ and $e_1 \wedge \ldots \wedge e_k \in \Lambda^k(M)$ by
$$ \langle f_1 \wedge \ldots \wedge f_k | e_1 \wedge \ldots \wedge e_k \rangle = \mbox{det}\paren{f_i(e_j)} $$
(and with $\langle F|E\rangle = 0$ if $F$ and $E$ belong to different exterior gradings).  It is easy to see that this pairing is well-defined, bilinear, and perfect in the sense that it effects an isomorphism of $\Lambda\paren{M^*}$ with $\Lambda\paren{M}^*$.  Since $(M^*)^*$ is canonically isomorphic to $M$, we of course have a pairing $\langle \cdot | \cdot \rangle: \Lambda\paren{M} \otimes \Lambda\paren{M^*} \rightarrow \mathbb{Z}$.

The pairings also allow for the interior product operation: for $E \in \Lambda\paren{M}$ and $F \in \Lambda\paren{M^*}$, let $\iota_E F \in \Lambda\paren{M^*}$ be defined by 
$$ \langle \iota_E F | \cdot \rangle = \langle F | E \wedge \cdot \rangle. $$
Again taking advantage of $(M^*)^* \cong M$, we also have an element $\iota_F E \in \Lambda\paren{M}$. 

Given an element $e \in M$, recall that the map given by wedging with $e$ turns $\Lambda(M)$ into an acyclic chain complex.  For an element $f \in M^*$, $\iota_f$ also turns $\Lambda(M)$ into an acyclic chain complex. 

If $f:M \rightarrow M'$ is a homomorphism, $f$ induces a homomorphism of exterior algebras, which we will also denote by $f$. 

We denote the set of integers $\{i, i+1, \ldots, j\}$ by $[i,j]$.
Given a basis $\{b_1, \ldots, b_r\}$ of $M$, and a subset $I = \{i_1,  \ldots, i_m\}$ of $[1,r]$ with $i_1 < i_2 < \ldots < i_m$, we write $b_I$ for $b_{i_1} \wedge \ldots \wedge b_{i_m}$.

\subsection{The sutured surface algebra}

\begin{defi}
\label{defi:sutalg}
For a ring $R$, the \emph{sutured surface algebra over $R$} for $(\Sigma, \mb{F})$ is
$$V(\Sigma, \mb{F}; R)^X = \Lambda\paren{H_1(\Sigma, \al^+;R)};$$
we let $V^i(\Sigma, \mb{F};R)^X$ denote $\Lambda^i\paren{H_1(\Sigma, \al^+;R)}$.
\end{defi}

We will usually drop the base ring from the notation, which will be understood to be $\Z$ or $\Z_2$ (the latter in the last two sections), and will always drop the $X$ (for ``exterior'') except when we compare our theory with the Sutured Floer Homology TQFT of \cite{HKM}, in Section \ref{sec:tqft}. 

The algebra $V(\Sigma, \mb{F})^*$ is identified with $$\Lambda\paren{H_1(\Sigma, \al^+)^*} = \Lambda\paren{H^1(\Sigma, \al^+)}.$$
As we noticed in Section \ref{sec:sutsur}, this may further be identified with  $\Lambda\paren{H_1(\Sigma, \al^-)}$.  

\subsection{Contact elements}

Before defining contact elements, we make a few more definitions.
When working over $\Z$, the theory on which we base our work suffers from unremovable sign ambiguities.  To cope with these, we have a strengthened notion of dividing sets.

\begin{defi}
A \emph{homology-oriented dividing set} $\veco{K}$ on a sutured surface is a pair $(K, \omega)$ where $K$ is a dividing set and $\omega$ is a choice of generator of $\Lambda^{top}\paren{H_1(R^+(K), \al^+; \Z)}$.
\end{defi}

To a dividing set $K$, there are two associated oriented dividing sets.  

Let $\pi_i: V(\Sigma, \mb{F}) \rightarrow V(\Sigma, \mb{F})$ denote the projection $V(\Sigma, \mb{F}) \rightarrow V^i(\Sigma, \mb{F})$, followed by the inclusion of $V^i(\Sigma, \mb{F})$ into $V(\Sigma, \mb{F})$.  (So, $\pi_i$ simply picks out the homogeneous part of an element in grading $i$.) Also, recall that $L(K) = n(\mb{F}) - \chi\paren{R^+(K)}$.

Finally, let $i_K$ denote the inclusion of $(R^+(K), \al^+)\hookrightarrow (\Sigma, \al^+)$, inducing a map 
$${i_K}_*:\Lambda\paren{H_1\paren{R^+(K), \al^+}} \rightarrow \Lambda\paren{H_1\paren{\Sigma, \al^+}} = V(\Sigma, \mb{F}).$$

\begin{defi}
\label{defi:contelt}
The \emph{contact element} of the oriented dividing set $\veco{K} = (K, \omega)$ on $(\Sigma, \mb{F})$ is given by
$$ c^X(\veco{K}) = \pi_{L(K)}\circ {i_K}_*(\omega) \in V(\Sigma, \mb{F})^X.$$
The \emph{contact subset} of the (unoriented) dividing set $K$, $[c^X(K)]$, is the set (of order 1 or 2)
$$ [c^X(K)] = \{c^X(\veco{K}) | \veco{K} \mbox{ is an oriented dividing set associated to }K \}.$$
\end{defi}

Again, we generally drop the $X$ from the notation. 
Note that if we work over $\Z_2$, we have no need for oriented dividing sets or contact subsets, and simply speak of contact elements (defined the same way) without ambiguity.

We immediately have the following, which explains some of the definition.

\begin{prop}
\label{prop:nontriv}
The contact element $c(\veco{K})$ is non-trivial if and only $K$ is non-isolating.
\end{prop}

\begin{proof}
The term ${i_K}_*(\omega)$ is non-trivial if and only if the inclusion $H_1\paren{R^+(K), \al^+} \rightarrow H_1(\Sigma, \al^+)$ is injective; according to Proposition \ref{prop:isocomps}, this occurs precisely when there are no isolated components of $R^-(K)$.  In this case, $c(\veco{K})$ is non-trivial exactly when ${i_K}_*(\omega)$ lies in exterior grading $L(K)$, and Proposition \ref{prop:isocomps} says that this occurs precisely when there are no isolated components of $R^+(K)$.
\end{proof}

\begin{remark}
Note that ${i_K}_*(\omega)$ is homogeneous with respect to the exterior grading, and so the projection $\pi_{L(K)}$ either preserves this element or takes it to $0$.  Essentially, the projection is part of the definition to artifically ensure that Proposition \ref{prop:nontriv} is true; of course, even if we had left out the projection, $c(\veco{K})$ would still be trivial when $R^-(K)$ had isolated components.
\end{remark}

We have biased all the above towards referring to $\al^+$ rather than $\al^-$, merely for the sake of definiteness.  We could imagine defining an element $c^-(\veco{K})$ in an exactly similar fashion.  It turns out that this is essentially equivalent information, but ``dualized'' in a sense; we treat this in Section \ref{sec:negative}.

\begin{figure}[!htbp]
	\centering
		\includegraphics[width=0.4\textwidth]{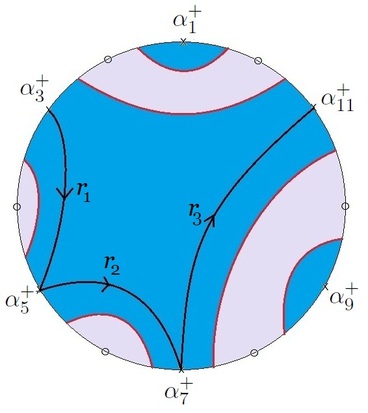}
	\caption{A dividing set $K$ on $\paren{D^2, \mb{F}(6)}$.}
	\label{fig:diskdiv}
\end{figure}

\subsection{Gluing morphisms}

The definition of gluing morphism is a bit trickier: note that in general, the image of $\al^+$ under $\phi_{\tau}$ includes some ``swallowed'' vertices in the interior of $\Sigma_{\tau}$.  In particular, $\phi_{\tau}$ induces a map from $H_1(\Sigma, \al^+)$ to $H_1\paren{\Sigma_{\tau}, \phi_{\tau}(\al^+)}$, rather than to $H_1\paren{\Sigma_{\tau}, \al^+_{\tau}}$, and so we cannot directly define gluing maps by the induced maps on homology.  However, $H_1(\Sigma_{\tau}, \al^+)$ can be thought of as a subgroup of $H_1\paren{\Sigma_{\tau}, \phi_{\tau}(\al^+)}$. On the level of exterior algebras, we can follow ${\phi_{\tau}}_*$ with an interior product which cuts away classes with boundary including swallowed vertices, and leaves an element in the appropriate subalgebra of $\Lambda\paren{H_1\paren{\Sigma_{\tau}, \phi_{\tau}(\al^+)}}$.

To be specific, we first further specify the notion of gluing.   
Given a gluing $\tau$ on $(\Sigma, \mb{F})$, consider the long exact cohomology sequence of the triple $(\Sigma_{\tau}, \phi_{\tau}(\al^+), \al^+_{\tau})$.  It is easy to see (noting the last demand of Definition \ref{defi:gluing}) that the coboundary map
$$ \delta: H^0(\phi_{\tau}(\al^+), \al^+_{\tau}) \rightarrow H^1(\Sigma_{\tau}, \phi_{\tau}(\al^+)) $$
is injective.  

\begin{defi}
An \emph{oriented gluing} $\veco{\tau}$ on a sutured surface is a pair $(\tau, \eta)$ where $\tau$ is an (unoriented) gluing and $\eta$ is a generator of $\Lambda^{top}\paren{\mbox{\emph{Im} }\delta}$.
\end{defi}

\begin{lemma}
\label{lem:images}
Let $\veco{\tau} = (\tau, \eta)$ be an oriented gluing on $(\Sigma, \mb{F})$. 
Let $j: H_1(\Sigma_{\tau}, \al^+_{\tau}) \rightarrow H_1\paren{\Sigma_{\tau}, \phi_{\tau}(\al^+)}$ denote the map induced by inclusion.  Then $j$ is injective. 

If 
$$\iota_{\eta}: \Lambda\paren{H_1\paren{\Sigma_{\tau}, \phi_{\tau}(\al^+)}} \rightarrow
\Lambda\paren{H_1\paren{\Sigma_{\tau}, \phi_{\tau}(\al^+)}} $$
is the interior product map, then
$$\mbox{\emph{Im} }\iota_{\eta} = \Lambda\paren{\mbox{\emph{Im} }j}$$
(the latter being a subalgebra of $\Lambda\paren{H_1(\Sigma_{\tau}, \phi_{\tau}(\al^+)}$).
\end{lemma}

\begin{proof}
Consider the long exact homology sequence of $(\Sigma_{\tau}, \phi_{\tau}(\al^+), \al^+_{\tau})$, which reads as
$$ 0 \rightarrow H_1(\Sigma_{\tau}, \al^+_{\tau}) \overset{j}{\rightarrow} H_1(\Sigma_{\tau}, \phi_{\tau}(\al^+)) \rightarrow H_0(\phi_{\tau}(\al^+), \al^+_{\tau}) \rightarrow 0.$$
The injectivity of $j$ is immediate; the image of $j$ consists of those classes of multicurves whose boundary lies in $\al^+_{\tau}$.  Since all the groups are free, this short exact sequence splits, and $\mbox{Im }j$ is exactly the subspace of $H_1\paren{\Sigma_{\tau}, \phi_{\tau}(\al^+)}$ annhilated by $\mbox{Im }\delta$.  

Choose a basis of $H_1\paren{\Sigma_{\tau}, \phi_{\tau}(\al^+)}$,$\{b_1, \ldots, b_k, b_{k+1}, \ldots, b_{\ell}\}$, such that $\{b_1, \ldots, b_k\}$ is a basis of $\mbox{Im }j$; also let $\{c_1, \ldots, c_{\ell}\}$ be the dual basis of $H^1\paren{\Sigma_{\tau}, \phi_{\tau}(\al^+)}$.
Of course, $\{c_{k+1}, \ldots, c_{\ell}\}$ is a basis of $\mbox{Im }\delta$, so $\eta = \pm c_{[k+1, \ell]}$.  It is not hard to see that 
$$\iota_{c_{[k+1,\ell]}}b_I = \left\{ \begin{array}{ll} \pm b_{I \setminus [k+1,\ell]}, &\mbox{ if }[k+1,\ell] \subset I\\ 0, &\mbox{ otherwise. }\end{array}\right.$$
From this, we can readily see that $\Lambda\paren{\mbox{Im }j}$ is exactly the image of $\iota_{\eta}$. 
\end{proof}

By Lemma \ref{lem:images}, the following definition makes sense.

\begin{defi}
\label{defi:gluemorph}
Given an oriented gluing $\veco{\tau} = (\tau, \eta)$ on a sutured surface $(\Sigma, \mb{F})$, the \emph{gluing morphism} 
$$ \Phi^X_{\veco{\tau}}: V(\Sigma, \mb{F})^X \rightarrow V(\Sigma_{\tau}, \mb{F}_{\tau})^X $$
is defined by
$$ \Phi^X_{\veco{\tau}}(x) = \iota_{\eta} \paren{{\phi_{\tau}}_*(x)};$$
here we are of course identifying $\Lambda\paren{\mbox{\emph{Im} }j}$ with $V(\Sigma_{\tau}, \mb{F}_{\tau})^X$.

Given an unoriented gluing $\tau$, we also write $[\Phi^X_{\tau}]$ for the pair of maps $\Phi^X_{\veco{\tau}}$ for the two oriented gluings $\veco{\tau}$ associated to $\tau$.
\end{defi}

\subsection{Examples}
\label{sec:notex}

We examine some examples in detail.  In Figure \ref{fig:diskdiv}, we show a dividing set $K$ on $\paren{D^2, \mb{F}(6)}$.  (Recall the notation set in Section \ref{sec:sutsur}.)  The positive region $R^+(K)$ is shown in blue, and $H_1(R^+(K), \al^+)$ is generated by the three arcs $r_1, r_2, r_3$.  Let $\veco{K} = (K, r_1 \wedge r_2 \wedge r_3)$.  Then 
$$c(\veco{K}) = \beta_3 \wedge \beta_5 \wedge (\beta_7 + \beta_9).$$

\begin{figure}[!htbp]
	\centering
		\includegraphics[width=0.45\textwidth]{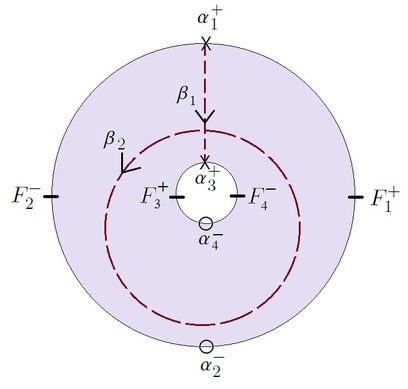}
	\caption{A sutured annulus $(S^1 \times D^1, \mb{F})$, with curves generating $H_1(S^1 \times D^1, \al^+)$.}
	\label{fig:annulus}
\end{figure}

In Figure \ref{fig:annulus}, we show a sutured annulus $(S^1 \times D^1, \mb{F})$, where $\mb{F}$ contains one positive vertex on each boundary component.  In this case, $H_1(S^1 \times D^1, \al^+)$ is generated by the arc $\beta_1$ and the core circle $\beta_2$ as shown.  Figure \ref{fig:annex} shows six examples of dividing sets on this annulus; with appropriately chosen homology orientations, we have
$$ c(\veco{K}_+) = 1,\,\,\, c(\veco{K}_-) = \beta_1 \wedge \beta_2,\,\,\,\, c(\veco{K}_0) = \beta_2,$$
$$ c(\veco{K}_1) = \beta_2,\,\,\, c(\veco{L}_0) = \beta_1,\,\,\,\, c(\veco{L}_1) = \beta_1 + \beta_2.$$

In Figure \ref{fig:glue2}, we show a disk $(D^2, \mb{F}(4))$ and a dividing set $K'$, such that if $\tau$ is the gluing which identifies the two vertical portions of the boundary, then $K'_{\tau}$ is the dividing set $K_-$ from above. Choosing a basis for homology as shown in the Figure, we have $c(\veco{K'}) = \gamma_1 \wedge \gamma_2 \wedge \gamma_3$ for one orientation on $K'$.   The map $\phi_{\tau}$ induces an isomorphism from $H_1(D^2, \al^+)$ to $H_1(S^1 \times D^1, \phi_{\tau}(\al^+))$, and we let $\gamma_1, \gamma_2, \gamma_3$ also denote a basis for the latter; then $\beta_1 = \gamma_1$ and $\beta_2 = \gamma_2 - \gamma_3$.   

An orientation on this gluing is supplied by $\eta$, where $\eta: H_1(S^1 \times D^1, \phi_{\tau}(\al^+)) \rightarrow \Z$ is given by evaluating the swallowed vertex on the boundary of a homology class -- or more simply, $\eta(\gamma_2) = \eta(\gamma_3) = 1$ and $\eta(\gamma_1)=0$.  Then,
$$\iota_{\eta}\paren{\gamma_1 \wedge \gamma_2 \wedge \gamma_3} = -\gamma_1 \wedge \gamma_3 + \gamma_1 \wedge \gamma_2 = \beta_1 \wedge \beta_2$$
as needed.

\begin{figure}[!htbp]
	\centering
		\includegraphics[width=0.65\textwidth]{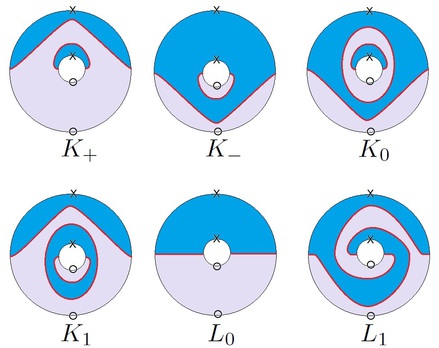}
	\caption{Six dividing sets on $(S^1 \times D^1, \mb{F})$.}
	\label{fig:annex}
\end{figure}

\begin{figure}[!htbp]
	\centering
		\includegraphics[width=0.4\textwidth]{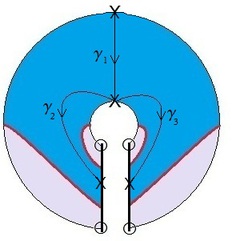}
	\caption{A dividing set $K'$ on $(D^2, \mb{F}(4))$, which induces $K_-$ when the two vertical portions of the boundary are identified.}
	\label{fig:glue2}
\end{figure}

\subsection{Gluing respects contact elements}

We address the axioms for a Sutured TQFT in the next section.  The heart of the notion is that the three assignments are coherent, in the following sense.

\begin{theorem}
\label{thm:respect}
If $\veco{K}$ is an oriented dividing set and $\veco{\tau}$ is an oriented gluing on a sutured surface $(\Sigma, \mb{F})$, then
$$ \Phi_{\veco{\tau}}\paren{c(\veco{K})} \in [c(K_{\tau})].$$
\end{theorem}

\begin{proof}

If $K$ is isolating, then so it $K_{\tau}$; in this case the result is clear.  So assume going forward that $K$ is non-isolating.

First we show that 
$$ {\phi_{\tau}}_*: H_1(R^+(K), \al^+) \rightarrow H_1\paren{R^+(K_{\tau}), \phi_{\tau}(\al^+)}$$
is an isomorphism.  Given a connected curve $\ell$ in $R^+(K_{\tau})$ with boundary in $\phi_{\tau}(\al^+)$, its preimage under $\phi_{\tau}$ will be a union of curves in $R^+(K)$ with boundary on $A^+$ (a boundary arc containing a positive vertex, as in Section \ref{sec:not}).  After a small isotopy of $\ell$, we can assume that its preimage actually lies in $\al^+$ itself.  So the map is surjective. 

To see that it is an isomorphism, we claim that the two groups have the same rank.  The rank of $H_1(R^+(K), \al^+)$ is $n(\mb{F}) - \chi\paren{R^+(K)} + I^+(K)$, as shown in Proposition \ref{prop:isocomps}.  Then, consider the exact sequence
$$0 \rightarrow H_1(R^+(K_{\tau})) \rightarrow H_1(R^+(K_{\tau}), \phi_{\tau}(\al^+)) \rightarrow H_0(\phi_{\tau}(\al^+)) \rightarrow H_0(R^+(K_{\tau})) \rightarrow H_0(\phi_{\tau}(\al^+)) \rightarrow 0.$$
Letting $r$ be the rank of the second group and $p$ be the number of pairs of points of $\al^+$ identified by the gluing, the ranks of the groups are $b_1\paren{R^+(K_{\tau})}, r, n(\mb{F}) - p, b_0\paren{R^+(K_{\tau})}, $ and $I^+(K)$ respectively. So
$$ r = n(\mb{F}) - \chi\paren{R^+(K_{\tau})} + I^+(K) - p.$$
Observe that each identified pair in $\al^+$ decreases the Euler characteristic by 1 from that of $R^+(K)$; thus, $r = n(\mb{F}) - \chi\paren{R^+(K)} + I^+(K)$ as needed.

Now, consider the diagram
$$\begin{CD}
H_1(R^+(K), \al^+) @>{\phi_{\tau}}_*>{\cong}> H_1\paren{R^+(K_{\tau}), \phi_{\tau}(\al^+)} @<<< H_1(R^+(K_{\tau}), \al^+_{\tau}) \\
@VVV   @VVV  @VVV\\
H_1(\Sigma, \al^+) @>{\phi_{\tau}}_*>> H_1\paren{\Sigma_{\tau}, \phi_{\tau}(\al^+)} @<<< H_1(\Sigma_{\tau}, \al^+_{\tau}) \\
\end{CD}$$
The lefthand commutative square shows that ${\phi_{\tau}}_*\paren{c(\veco{K})} \in \Lambda\paren{H_1\paren{\Sigma_{\tau}, \phi_{\tau}(\al^+)}}$ is equal to the image of a generator $\Lambda^{top}\paren{H_1\paren{R^+(K_{\tau}), \phi_{\tau}(\al^+)}}$.  Note also that the two horizontal maps in the righthand square are injective.  

The rightmost vertical map in the diagram is injective if and only if $K_{\tau}$ has no isolated negative components, by Proposition \ref{prop:isocomps}.  If this map is \emph{not} injective, since the horizontal maps in the righthand square are injective, the middle vertical map is not injective either. So in this case, ${\phi_{\tau}}_*\paren{c(\veco{K})}$ and hence $\Phi_{\veco{\tau}}\paren{c(\veco{K})}$ are zero, as we should have.  

Otherwise, let $\{b_1, \ldots, b_k, b_{k+1}, \ldots, b_{\ell}, b_{\ell+1}, \ldots, b_m\}$ be a basis of $H_1\paren{\Sigma_{\tau}, \phi_{\tau}(\al^+)}$ such that the first $k$ elements are a basis of the image of $H_1(R^+(K_{\tau}), \al^+_{\tau})$ and the first $\ell$ elements are a basis of $\mbox{Im }j$ (as in Lemma \ref{lem:images}).  Also, let $\{c_1, \ldots, c_m\}$ be the dual basis.
As in the proof of Lemma \ref{lem:images}, $\eta$ will be $\pm c_{[\ell+1,m]}$.  A representative for $[c(K_{\tau})]$ will be given by either $0$ or $b_{[1,k]}$, depending on whether or not $K_{\tau}$ has isolated positive components.    

Consider the exact sequence 
$$ 0 \raw H_1(R^+(K_{\tau}), \al^+_{\tau}) \raw H_1\paren{R^+(K_{\tau}), \phi_{\tau}(\al^+)} \raw H_0(\phi_{\tau}(\al^+), \al^+_{\tau})$$
$$ \raw H_0(R^+(K_{\tau}), \al^+_{\tau}) \raw H_0\paren{R^+(K_{\tau}), \phi_{\tau}(\al^+)} \raw 0. $$
The ranks of the last two groups are the number of isolated components of $R^+(K_{\tau})$ and $R^+(K)$, respectively; so, under the assumption that $K$ is non-isolating, the last term drops out.  If $K_{\tau}$ has isolated positive components, then the boundary map from $H_1\paren{R^+(K_{\tau}), \phi_{\tau}(\al^+)}$ to $H_0(\phi_{\tau}(\al^+), \al^+_{\tau})$ will certainly not be surjective.  Working in the above basis, it is easy to see that in this case, $\iota_{\eta}\paren{{\phi_{\tau}}_*\paren{c(\veco{K})}}$ will be trivial, as it should be.  

Finally, if $K_{\tau}$ (and hence $K$) are non-isolating, then the exact sequence above shortens to
$$  0 \raw H_1(R^+(K_{\tau}), \al^+_{\tau}) \raw H_1\paren{R^+(K_{\tau}), \phi_{\tau}(\al^+)} \raw H_0(\phi_{\tau}(\al^+), \al^+_{\tau}) \raw 0. $$
The image of $H_1(R^+(K_{\tau}), \al^+_{\tau})$ in $H_1\paren{\Sigma_{\tau}, \phi_{\tau}(\al^+)}$ will therefore be the span of the set \linebreak $\{b_1, \ldots, b_k, b_{\ell + 1}, \ldots, b_m\}$.  Working in the above basis, it is not hard to see that
$$ \iota_{\eta}\paren{{\phi_{\tau}}_*\paren{c(\veco{K})}} = \pm b_{[1,k]};$$
so $\Phi_{\veco{\tau}}\paren{c(\veco{K})}$ will be a representative of $[c(K_{\tau})]$.

\end{proof}

\section{Sutured TQFT, and the SFH and Exterior TQFTs}
\label{sec:tqft}

In this section, we make precise what we mean by a Sutured TQFT.  We show that our construction provides an example. After noting a few similarities of our construction with the Sutured Floer Homology version, we proceed to show that these construction are equivalent, in an appropriate sense.

\subsection{Axioms for Sutured TQFT}
We translate the axioms satisfied by the sutured quantum field theory of \cite{HKM} to the current setting, gently altering them to fit our setup.  Again, we work over $\Z$, with obvious simplifications if we work over $\Z_2$.

\begin{defi}
\label{defi:tqft}
A \emph{sutured topological quantum field theory} $\mathcal{F}$ assigns 

\begin{itemize}
\item[I.] to each sutured surface $(\Sigma, \mb{F})$, a free $\Z$-module $V(\Sigma, \mb{F})^{\mathcal{F}}$, the \emph{sutured surface group} of $(\Sigma, \mb{F})$;

\item[II.] to each dividing set $K$ on $(\Sigma, \mb{F})$, a subset $[c^{\mathcal{F}}(K)] \subset V(\Sigma, \mb{F})^{\mathcal{F}}$ of the form $\{c, -c\}$, the \emph{contact subset} of $K$;

\item[III.] to each gluing $\tau$ on $(\Sigma, \mb{F})$, a set of maps $[ \Phi^{\mathcal{F}}_{\tau}]$ from $V(\Sigma, \mb{F})^{\mathcal{F}}$ to $V(\Sigma_{\tau}, \mb{F}_{\tau})^{\mathcal{F}}$, of the form $\{\Phi, -\Phi\}$, the \emph{gluing morphism set} of $\tau$,
\end{itemize}

\noindent subject to the following conditions.  (In the following, we drop all $\mathcal{F}$ superscripts, and let $c(\cdot)$ and $\Phi_{\tau}$ denote either element of their respective sets.)

\begin{enumerate}

\item $V(\Sigma, \mb{F})$ supports a $\Z$-grading (the \emph{$\mbox{Spin}^c$-grading}) such that the homogeneous submodule in grading $L(\Sigma, \mb{F})-2i$ is free of rank $\left(\hspace{-6pt}\begin{array}{c} L(\Sigma, \mb{F}) \\ i \end{array}\hspace{-6pt}\right)$ for $0 \leq i \leq L(\Sigma, \mb{F})$, and trivial for other gradings. In particular, the rank of the entire module is $2^{L(\Sigma, \mb{F})}$.  

\item If $(\Sigma, \mb{F}) = (\Sigma_1, \mb{F}_1) \coprod (\Sigma_2, \mb{F}_2)$, then there exists an isomorphism
$$ \varphi: V(\Sigma_1, \mb{F}_1) \otimes V(\Sigma_2, \mb{F}_2) \rightarrow V(\Sigma, \mb{F})$$
such that for $K_i$ a dividing set on $(\Sigma, \mb{F}_i)$
$$ \varphi\left(c(K_1) \otimes c(K_2)\right) = \pm c(K_1 \sqcup K_2).$$

\item If $K$ has a homotopically trivial closed component, then $c(K) = 0$.

\item Gluing respects contact subsets: that is,
$$\Phi_{\tau}\left(c(K)\right) = \pm c(K_{\tau}).$$

\item The assignments are topologically invariant. Specifically, if $f:(\Sigma, \mb{F}) \rightarrow (\Sigma', \mb{F}')$ is a diffeomorphism, then there exists an isomorphism $f_*:V(\Sigma, \mb{F}) \rightarrow V(\Sigma', \mb{F}')$ such that 
$$ f_*\left(c(K)\right) = \pm c\left(f(K) \right)$$
for all dividing sets $K$.  Also, if $\tau$ is a gluing on $(\Sigma, \mb{F})$, $\tau' = f \circ \tau \circ f^{-1}$ is the corresponding gluing on $(\Sigma', \mb{F}')$, and $g:(\Sigma_{\tau}, \mb{F}_{\tau}) \rightarrow (\Sigma'_{\tau'}, \mb{F}'_{\tau'})$ the induced diffeomorphism, then there exists an isomorphism 
$$g_*:V(\Sigma_{\tau}, \mb{F}_{\tau}) \rightarrow V(\Sigma'_{\tau'}, \mb{F}'_{\tau'})$$ 
which respects contact elements and satisfies 
$$ g_* \circ \Phi_{\tau} = \Phi_{\tau'} \circ f_*.$$
\end{enumerate}
\end{defi}

The definition of a sutured TQFT was designed to fit the Sutured Floer Homology of a product of a sutured surface with $S^1$, and the product contact structure associated to a dividing set on a sutured surface.  Precisely, $V(\Sigma, \mb{F})^{SFH}$ is defined as $SFH\paren{-(S^1 \times \Sigma), -(S^1 \times \al)}$, $[c^{SFH}(K) ]$ is the contact subset of the product contact structure associated to $K$, and $[ \Phi^{SFH}_{\tau}]$ is the map (defined up to multiplication by $\pm 1$) on Sutured Floer homology defined in \cite{HKM}, where $\paren{-(S^1 \times \Sigma), -(S^1 \times \al)}$ is treated as a submanifold of $\paren{-(S^1 \times \Sigma_{\tau}), -(S^1 \times \al_{\tau})}$.  We denote this TQFT by $\mathcal{SFH}$.  (The reasons for the minus signs, the usage of $\al$ in place of $F$, and much more is explained in \cite{HKM}.)  Of course, we may also make these constructions over $\Z_2$, which we denote by $\mathcal{SFH}(\Z_2)$. 

The last condition is left implicit in \cite{HKM}.
We anticipate needing to consider automorphisms of particular sutured surface in future work, and so we make the last condition explicit here.  Of course, in both our TQFT and the SFH one, topological invariance is clear.  (Note that we do not ask for any sort of uniqueness for the isomorphisms $f_*$, just that some isomorphism exists.)

\begin{theorem}
\label{thm:mainthm}
The assignments $V(\Sigma, \mb{F})^X$, $[c^X(K)]$, and $[\Phi^X_{\tau}]$ form a sutured TQFT.
\end{theorem}

We denote this TQFT by $\mathcal{X}$, and the analogous construction over $\Z_2$ by $\mathcal{X}(\Z_2)$.
 
\begin{proof}
We again drop superscripts.  Set the $\mbox{Spin}^c$-grading on $V(\Sigma, \mb{F})$ so that $V^i(\Sigma, \mb{F})$ is the homogeneous submodule of grading $L(\Sigma, \mb{F}) - 2i$.  With this, condition (1) is clear.  Condition (2) is equivalent to the statement, for two free $\Z$-modules $M_1, M_2$, that $\Lambda(M_1 \oplus M_2) \cong \Lambda(M_1) \otimes \Lambda(M_2)$.  Condition (3) is a special case of Proposition \ref{prop:nontriv}.  Condition (4) is Theorem \ref{thm:respect}.  Since everything in sight is defined via intrinsic topological invariants, condition (5) is clear.
\end{proof}

\subsection{Comparisons between the TQFTs}

In this subsection, we drop the superscript $X$'s; we always refer to the exterior TQFT, although analogous properties also hold for the SFH TQFT.

First, we note that the following is immediate from definitions and Proposition \ref{prop:nontriv}.

\begin{prop}
Let $\veco{K}$ be a oriented dividing set on a sutured surface $(\Sigma, \mb{F})$.  Then the following are equivalent:
\begin{enumerate}

\item $K$ is non-isolating;

\item $c(\veco{K}) \in V(\Sigma, \mb{F})$ is non-zero;

\item $c(\veco{K}) \in V(\Sigma, \mb{F})$ is primitive.

\end{enumerate}
\end{prop}

The analogous claim for the Sutured Floer Homology TQFT was Conjecture 7.13 of \cite{HKM}, and was proved in \cite{Massot}.

Next, a \emph{simple gluing} is a gluing $\tau$ where the glued arcs $\gamma$ and $\gamma'$ each contain a single suture.  The following should be compared with Lemma 7.9 of \cite{HKM}.

\begin{prop}
\label{prop:simpglue}
If $\veco{\tau}$ is a simple gluing equipped with orientation, then $\Phi_{\veco{\tau}}$ is an isomorphism.
\end{prop}

\begin{proof}
Since there are no vertices in the interior of $\gamma$ and $\gamma'$, the set $\phi_{\tau}(\al^+) \setminus \al^+_{\tau}$ is empty.  So, $\phi_{\tau}$ induces a map from $H_1(\Sigma, \al^+) \rightarrow H_1(\Sigma_{\tau}, \al^+_{\tau})$.  It is easy to see, as in the proof of Theorem \ref{thm:respect}, that this map is surjective.  Furthermore, in the glued surface, the Euler characteristic goes down by 1 (since we glue along a single non-closed arc), but the number of positive vertices also goes down by one (the vertices on $\gamma$ and $\gamma'$ becoming a single vertex).  By Lemma \ref{lem:lrank}, the groups are therefore isomorphic, meaning that the map is an isomorphism.
\end{proof}

In \cite{HKM}, some time is spent looking at the sutured annulus $(S^1 \times D^1, \mb{F})$ considered in Section \ref{sec:notex}.  The non-isolating dividing sets on this surface are denoted $K_+, K_-, K_0, K_1$ and $L_n$ for $n\in \Z$. The first four, as well as $L_0$ and $L_1$, are those shown in Figure \ref{fig:annex}; $L_n$ for other values of $n$ is gotten by applying the diffeomorphism $T^n$ to $L_0$, where $T$ is a positive Dehn twist about a core circle of the annulus.  An analogue of the following was conjectured in Section 7.5 of \cite{HKM}.

\begin{prop}
There exist $x, y \in V(\Sigma, \mb{F})$ such that
$$ x \in [ c(L_0)],\,\,\,\, y \in [ c(K_0) ] = [c(K_1) ],\,\,\,\,\mbox{and } x + ny \in [c(L_n)].$$
\end{prop}

\begin{proof}
Referring back to Section \ref{sec:notex}, let $x = \beta_1$ and $y = \beta_2$.  Since $T_*(\beta_1) = \beta_1+\beta_2$ and $T_*(\beta_2) = \beta_2$, we have $\paren{T^n}_*(\beta_1) = \beta_1 + n\beta_2$.
\end{proof}

Finally, we end this subsection with a definition and a result that will be important to us.  Let $L_1, L_2, $ and $L_3$ be the three dividing sets shown in Figure \ref{fig:byp1} on $\paren{D^2, \mb{F}(3)}$.  On a general sutured surface $(\Sigma, \mb{F})$, an unordered triple of dividing sets $\{K_1, K_2, K_3\}$ is said to form a \emph{bypass triple} if there is a sutured surface $(\Sigma', \mb{F}')$, a dividing set $K'$ on $(\Sigma', \mb{F}')$, and a gluing $\tau$ on $\paren{\Sigma' \coprod D^2, \mb{F}' \coprod \mb{F}(3)}$, such that $\paren{\Sigma' \coprod D^2}_{\tau} = \Sigma$, $\paren{\mb{F}' \coprod \mb{F}(3)}_{\tau} = \mb{F}$ and $\paren{K' \coprod L_i}_{\tau} = K_i$ for $i=1,2,3$.  Slightly less formally, $\{K_1, K_2, K_3\}$ form a bypass triple if they are the same outside of some disk in the interior of $\Sigma$, and within this disk all three contain three parallel arcs, but differentiated by $120^{\circ}$ twists, as shown in Figure \ref{fig:byptrip}.  The contact subsets of three dividing sets in a bypass triple obey the following relation.

\begin{figure}[htbp]
	\centering
		\includegraphics[width=0.65\textwidth]{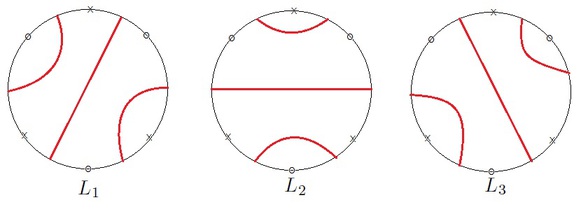}
	\caption{The three dividing sets $L_1, L_2, L_3$ on $\paren{D^2, \mb{F}(3)}$.}
	\label{fig:byp1}
\end{figure}

\begin{figure}
	\centering
		\includegraphics[width=0.75\textwidth]{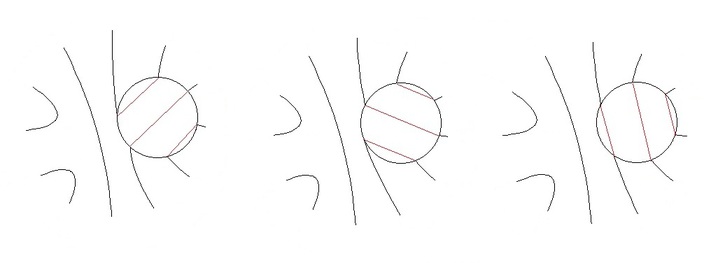}
	\caption{Three dividing sets that are bypass related; they are the same outside of the small disk shown.}
	\label{fig:byptrip}
\end{figure}

\begin{prop}
\label{prop:byp}
If $\{K_1, K_2,K_3\}$ is a bypass triple of dividing sets on a sutured surface $(\Sigma, \mb{F})$, then for any choices of orientations on these dividing sets, there exist $\epsilon_1, \epsilon_2, \epsilon_3 \in \{\pm 1\}$ such that
$$ \epsilon_1c(\veco{K}_1) + \epsilon_2c(\veco{K}_2) +\epsilon_3c(\veco{K}_3) = 0.$$
\end{prop}

\begin{proof}
The equation clearly holds for the special case of $(\Sigma, \mb{F}) = (D^2, \mb{F}(3))$ with $\veco{K}_i$ associated to the non-isolating dividing set $L_i$ as above, for $i=1,2,3$.   For a bypass triple $\{K_1, K_2, K_3\}$ on a general surface $(\Sigma, \mb{F})$, simply note that $c(\veco{K}_i) = \pm\Phi_{\veco{\tau}}\paren{c(\veco{K}') \otimes c(\veco{L}_i)}$ with $\tau$ and $K'$ as above, and any homology orientations on dividing sets and gluings.
\end{proof}

The analogue of this result in the Sutured Floer homology TQFT is discussed in Section 7.2 of \cite{HKM}.

\subsection{Equivalence of the SFH and Exterior TQFTs}
\label{sec:atss}

In light of the above, it is natural to ask whether our TQFT is equivalent to that of \cite{HKM}.  To make the question precise, we have the following.

\begin{defi}
\label{defi:equiv}
We say that two sutured TQFTs $\mathcal{F}$ and $\mathcal{G}$ are \emph{equivalent} if there are isomorphisms 
$$H_{\Sigma, \mb{F}}: V(\Sigma, \mb{F})^{\mathcal{F}} \rightarrow V(\Sigma, \mb{F})^{\mathcal{G}}$$ 
for every sutured surface $(\Sigma, \mb{F})$, such that

\begin{enumerate}

\item $H_{\Sigma, \mb{F}}\left([c^{\mathcal{F}}(K)]\right) = [c^{\mathcal{G}}(K)]$ for all dividing sets $K$ on $(\Sigma, \mb{F})$; and

\item for any gluing $\tau$ on $(\Sigma, \mb{F})$, 
$$ H_{\Sigma_{\tau}, \mb{F}_{\tau}} \circ \Phi^{\mathcal{F}}_{\tau} = \pm \Phi^{\mathcal{G}}_{\tau} \circ H_{\Sigma, \mb{F}}$$
where the maps $\Phi^{\mathcal{F}}_{\tau}$ and $\Phi^{\mathcal{G}}_{\tau}$ are any representatives of the gluing morphism sets.

\end{enumerate}
\end{defi}

This brings us to our second main theorem.

\begin{theorem}
\label{thm:isom}
$\mathcal{SFH}(\Z_2)$ and $\mathcal{X}(\Z_2)$  are equivalent (in the sense of Definition \ref{defi:equiv}).
\end{theorem}

We expect that the analogous statement with coefficients in $\Z$ is true as well.

The rest of this section is devoted to the proof of Theorem \ref{thm:isom}.  As we work over $\Z_2$, we may totally dispense with homology orientations and all that come with them.   
Note also that we can then write Proposition \ref{prop:byp} as
$$ c(K_1) + c(K_2) + c(K_3) = 0$$
for a bypass triple $\{K_1, K_2, K_3\}$.

The strategy we undertake to show the equivalence is to cut our surface into copies of $\paren{D^2, \mb{F}(2)}$. (This process is studied in depth in \cite{itsybitsy}.)  When we are done, we will have in fact shown that for any TQFT over $\Z_2$, the contact subsets for dividing sets on $\paren{D^2, \mb{F}(2)}$ essentially determine all the information contained in the TQFT.

Recall notation from Section \ref{sec:sutsur}.   
Up to isotopy, there are two non-isolating dividing sets $K_1$ and $K_2$ on $\paren{D^2, \mb{F}(2)}$: $K_1$ has a chord connecting $F_1$ with $F_2$, and $K_2$ has a chord connecting $F_1$ with $F_4$. 
The group $H_1(R^+(K_1), \al^+)$ is trivial, while $H_1(R^+(K_2), \al^+) \cong H_1(\Sigma, \al^+)$ is generated by the class of $\beta_1$.  Then $V\paren{D^2, \mb{F}(2)} \cong \Z_2^2$ is generated by $c\paren{K_1} = 1$ and $c\paren{K_2} = \beta_1$.

Call a sutured surface $(\Sigma, \mb{F})$ \emph{totally decomposed} if $(\Sigma, \mb{F}) = \bigsqcup_{i=1}^k (\Sigma_i, \mb{F}_i)$ for some $k$, where each $(\Sigma_i, \mb{F}_i)$ diffeomorphic to $\paren{D^2, \mb{F}(2)}$.  Then, if we fix diffeomorphisms identifying each component of $(\Sigma, \mb{F})$ with this prototype, it is clear that each dividing set is determined by its contact element, which can be written as a tensor product of elements all of which are $1$ or $\beta_1$.  

Let $\ell$ be a multicurve in $\Sigma$ composed of arcs that intersect only at their boundaries, with one boundary point of each component in $\al^+$ and the other boundary in $\al^-$.  We may cut open $\Sigma$ along $\ell$ to form a new surface, which we denote $\Sigma_{\ell}$.  Let $\phi_{\ell}$ denote the quotient map that is the reverse of cutting open.  Then $\Sigma_{\ell}$ naturally comes equipped with sutures $F_{\ell}^{\pm}$ and vertices $\al_{\ell}^{\pm}$.  Specifically, $\al_{\ell}^+ = \phi_{\ell}^{-1}(\al^+)$, and similarly for $\al_{\ell}^-$; notice that the endpoints of $\ell$ become doubled.  Similarly define $F_{\ell}^{\pm}$, except each with one new suture in the interior of $\phi_{\ell}^{-1}(\ell)$ for each component of $\ell$, with signs determined by the neighboring vertices.  The reverse of the process of cutting open a surface is a gluing on $(\Sigma_{\ell}, \mb{F}_{\ell})$.

\begin{prop}
\label{prop:cutitdown}
Let $(\Sigma, \mb{F})$ be a sutured surface with no components diffeomorphic to $(D^2, \mb{F}(1))$.  Then there is a totally decomposed sutured surface $(\Sigma', \mb{F}')$, and a gluing $\tau$ on $(\Sigma', \mb{F}')$, such that $(\Sigma, \mb{F}) \cong ({\Sigma'}_{\tau}, {\mb{F}'}_{\tau})$ and $\Phi_{\tau}$ is an isomorphism.
\end{prop}

\begin{proof}
Assume for simplicity that $\Sigma$ is connected. We claim that we can \emph{quadrangulate} $(\Sigma, \mb{F})$: that is, we can put a cellular structure on $\Sigma$ such that:

\begin{enumerate}
\item all $0$-cells lie in $\al$;

\item all $1$-cells run from a point of $\al^+$ to a point of $\al^-$; and

\item all $2$-cells have boundary composed of four 1-cells.
\end{enumerate}

We proceed by induction on the genus of $\Sigma$.
The claim is clear for sutured disks.  For sutured punctured disks, start by taking $1$-cells running from a fixed point in $\al^+$ on one boundary component to points in $\al^-$ in each of the other boundary components.  Now cut our punctured disk open along these $1$-cells.  This yields an unpunctured sutured disk, which we can quadrangulate, and this quadrangulation can be lifted to a quadrangulation of the punctured disk.  So we are done in the genus $0$ case.

Suppose we have the claim for genus $g-1$.
If the genus of $\Sigma$ is $g > 0$, take a simple closed curve $C$ with $[C]$ non-trivial and not in the image of $H_1(\partial \Sigma)$.  Let $C'$ be an arc that starts on $\al^+$ and ends on an adjacent vertex in $\al^-$, isotopic to a short arc along $\partial \Sigma$.  Isotope $C$ so that a single point touches $C'$ tangentially, and then resolve this intersection to obtain a single arc $\ell$.  Now set this to be a $1$-cell.  Cutting open $\Sigma$ along $\ell$ and appealing to the inductive hypothesis, we can find a quadrangulation on $(\Sigma, \mb{F})$ as desired.

Having shown the claim, let $(\Sigma', \mb{F}')$ be the sutured surface gotten by cutting open $(\Sigma, \mb{F})$ along the $1$-skeleton of a quadrangulation.  This surface is totally decomposed.  It only remains to show that the reverse of this cutting is a gluing $\tau$ whose morphism is an isomorphism.  

Notice that $\phi_{\tau}(\al^+) = \al^+_{\tau}$, since every vertex of $(\Sigma', \mb{F}')$ maps to a $0$-cell of the quadrangulation, and all $0$-cells are on the boundary.  In particular, $H_0(\phi_{\tau}(\al^+), \al^+_{\tau})$ is trivial; hence $\Phi_{\tau}$ is just the map 
$$ {\phi_{\tau}}_*: \Lambda\paren{H_1(\Sigma', {\al'}^+)} \rightarrow \Lambda\paren{H_1(\Sigma, {\al}^+)}.$$
Thus, it suffices to show that ${\phi_{\tau}}_*$ (as a map on homology) is an isomorphism.  It is easy to see that $H_1(\Sigma', {\al'}^+)$ and $H_1(\Sigma, {\al}^+)$ have the same ranks, using Lemma \ref{lem:lrank}.  Surjectivity is also easy to see by looking at preimages as before.  Therefore, $\Phi_{\tau}$ is an isomorphism.
\end{proof}

\begin{proof}[Proof of Theorem \ref{thm:isom}]
For simplicity, assume that all surfaces $(\Sigma, \mb{F})$ have no components isomorphic to $\paren{D^2, \mb{F}(1)}$; it is easy to modify the proof in the case where there are. 

First, consider the case of $\paren{D^2, \mb{F}(2)}$.
In this case, $V\paren{D^2,\mb{F}(2)}^X$ and $V\paren{D^2, \mb{F}(2)}^{SFH}$  are each rank 2, but considering the maps as graded maps, each homogeneous level is of rank 1, so the map $H_{D^2, \mb{F}(2)}$ is determined.  More generally, for a totally decomposed sutured surface $(\Sigma, \mb{F})$, there is clearly a unique map $H_{\Sigma, \mb{F}}$ which respects contact elements and the tensor product decomposition.

Now, by Proposition \ref{prop:cutitdown}, any sutured surface $(\Sigma, \mb{F})$ can be realized as $({\Sigma'}_{\tau_0}, {\mb{F}'}_{\tau_0})$ for a totally decomposed surface $(\Sigma', \mb{F}')$ with a gluing $\tau_0$ for which $\Phi^X_{\tau_0}$ is an isomorphism.  By an easy application of Lemma 7.9 of \cite{HKM}, the morphism 
$$ \Phi^{SFH}_{\tau_0}: V\paren{{\Sigma'}, {\mb{F}'}}^{SFH} \rightarrow V\paren{\Sigma, \mb{F}}^{SFH}$$
must be an isomorphism which respects contact elements.  Indeed, up to diffeomorphism, $\tau_0$ can be realized as a succession of simple gluings, each one of whose morphisms must be an isomorphism respecting contact elements; since there is a basis of contact elements in  $V\paren{{\Sigma'}, {\mb{F}'}}^{SFH}$, and we know $\Phi^{SFH}_{\tau_0}$ must also respect this basis, it follows that  $\Phi^{SFH}_{\tau_0}$ must coincide with the composition of morphisms.

Now, we can let
$$H_{\Sigma, \mb{F}} = \Phi^X_{\tau_0} \circ H_{\Sigma', \mb{F}'} \circ \paren{\Phi^{SFH}_{\tau_0}}^{-1}: V(\Sigma, \mb{F})^{SFH} \rightarrow V(\Sigma, \mb{F})^X.$$
This map, which we now shorten to $H$, is again an isomorphism, which respects contact elements for any dividing set on $(\Sigma, \mb{F})$ induced by $\tau_0$.

It still remains to show that $H$ respects \emph{all} contact elements.  Let $S$ denote the $1$-skeleton of the quadrangulation used to produce $(\Sigma', \mb{F}')$, as in the proof of Proposition \ref{prop:cutitdown}.  Then a dividing set $K$ on $(\Sigma, \mb{F})$ is not induced by the reverse gluing $\tau_0$ when $K$ intersects some arc of $S$ more than once.  

We proceed by induction on the number of excess intersections of $K$ with $S$ -- that is, the number of intersections of $K$ with $S$ minus the number of components of $S$ which intersect $K$.  Note that this number must be even, due to the demand that the arcs of $S$ start and end on vertices of opposite sign.  

We have shown that $H$ satisfies the desired condition when there are no excess intersections.  So, suppose that $ H\paren{c^{SFH}(K)} = c^X(K)$ for all $K$ with no more than $c$ excess intersections with $S$.  Let $K_0$ be a dividing set which has $c+2$ intersections with $S$.  Then, take a disk from $\Sigma$ that straddles a component of $S$ and intersects $K_0$ in three parallel arcs, and form $K_1$ and $K_2$ by altering $K_0$ within this disk so that the three dividing sets form a bypass triple.  After an isotopy, $K_1$ and $K_2$ will have at most $c$ excess intersections with $S$.  

So $ H\paren{c^{SFH}(K_i)} = c^X(K_i)$ for $i=1,2$.  Over $\Z_2$, it is shown in \cite{HKM} that $c^{SFH}(K_0) = c^{SFH}(K_1) + c^{SFH}(K_2)$.  Meanwhile, the corresponding relation $c^X(K_0) = c^X(K_1)+c^X(K_2)$ holds for our TQFT by Proposition \ref{prop:byp}.  Therefore, $ H\paren{c^{SFH}(K_0)} = c^X(K_0)$. 

To finish demonstrating equivalence, we need to show that the equivalence maps commute with gluing, which is equivalent to the statement that
$$ \Phi^{SFH}_{\tau} = H^{-1}_{\Sigma_{\tau}, \mb{F}_{\tau}} \circ \Phi^X_{\tau} \circ H_{\Sigma, \mb{F}}$$
for any gluing $\tau$.  To see this, first observe that for any surface $(\Sigma, \mb{F})$, there is a basis of $V(\Sigma, \mb{F})^{SFH}$ composed of the contact elements of some dividing sets $K_1, K_2, \ldots, K_{M}$, where $M = 2^{L(\Sigma, \mb{F})}$; indeed, take all the dividing sets induced by the gluing $\tau_0$ used above.  Now, we simply note that both maps take $c^{SFH}(K_i)$ to $c^{SFH}({K_i}_{\tau})$ for all $i$, since both respect contact elements.
\end{proof}

Notice that we used very little about Sutured Floer Homology in the course of the proof.  It is easy to see that in fact, the above argument shows the following.

\begin{corollary}
\label{cor:uniquez2}
Let $\mathcal{F}$ be a sutured TQFT over $\Z_2$, for which

\begin{enumerate}

\item $\Phi_{\tau}$ is an isomorphism for any simple gluing $\tau$; and

\item there are isomorphisms $f_i: V(D^2, \mb{F}(i))^{\mathcal{F}} \rightarrow V(D^2, \mb{F}(i))^X$ which respect contact elements, for $i=1,2$.

\end{enumerate}
Then $\mathcal{F}$ is equivalent to $\mathcal{X}(\Z_2)$.
\end{corollary}

\begin{proof}
The only facts we used about $\mathcal{SFH}$ in the proof were the above two items, general properties of any sutured TQFT, and the bypass relation, which itself can be deduced from these two properties.
\end{proof}

\begin{remark}
The argument of Theorem \ref{thm:isom} needs to be considerably strengthened if we wish to prove the analogue over $\Z$.  We can construct isomorphisms between the theories for any two surfaces, which respect bases composed of contact elements, as above; however, it will not be clear that these isomorphisms will be unique up to a single, overall sign -- which we need, since we use additive structure in showing that the isomorphisms respect all contact elements.  In the author's experience, most ``clever'' attempts to extend the argument to $\Z$ run up against this or a similar issue at some point; it seems like some careful work, at least for the case of disks, is necessary to extend the argument. 
\end{remark}

\section{Sutured Disks, Giroux's Criterion on $S^2$, and Tight Solid Tori}
\label{sec:disks}

We now study the special case of sutured disks in detail.  We also show that there is a curious interpretation of the product operation implicit in our construction.  We again work with $\Z_2$-coefficients for simplicity, although everything here can clearly be done over $\Z$ (the cost being a cluttering of notation).  At the end of this section, we apply this to a result that is useful for the classification of tight contact structures on solid tori with convex boundary.  This result could in fact be modified without much difficulty to address more general handlebodies, via the ``state transition'' picture described in \cite{HondaII}.

Much of this section should be compared with \cite{Mathews}. 

\subsection{Sutured disks}
 
The dividing sets on $(D^2, \mb{F}(N))$ are isotopy classes of embeddings of $n$ arcs into $D^2$ with boundaries lying in $F$, such that no arcs intersect and each point in $F$ is at the end of a single arc.  It is easy to see that such isotopy classes determined by the data of which points of $F$ share an arc.  The number of such isotopy classes is well known to be given by the $N$th Catalan number, $\frac{1}{N+1}\paren{\begin{array}{c}2N\\ N\end{array}}$.  

We first show directly that all non-isolating dividing sets on a sutured disk are distinguished by their contact elements; note that this fact also follows by a direct application of Theorem \ref{thm:isom} together with Proposition 7.5 of \cite{HKM}.

\begin{prop}
If $K$ and $K'$ are non-isolating dividing sets on $(D^2, \mb{F}(N))$ that are not isotopic, then $c(K) \neq c(K')$. 
\end{prop}

\begin{proof}
Without loss of generality, assume that $K$ has an arc running from $F_1$ to $F_i$, and that $K'$ has an arc running from $F_1$ to $F_j$, $i < j$.  Consider the oriented curve $\gamma_0$ running from $\al_{2N}$ to $F_1$, along the aforementioned arc of $K$, and then from $F_i$ to $\al_i$.  This curve can be isotoped to an oriented curve $\gamma$ from $\al_{2N}$ to $\al_i$ which lies entirely within $R^-(K)$.  So $\iota_{\gamma}c(K) = 0$.

On the other hand, there is a curve $\gamma'$ from $\al_1$ to $\al_{j-1}$, constructed in a similiar fashion,  which lies within $R^+(K')$ and intersects $\gamma$ once transversally, since $1<i<j<2N$.  Complete $\gamma'$ to a basis $\{\gamma', b_2, \ldots, b_k\}$ of $H_1\paren{R^+(K'),\al^+}$, which we make think of as a subspace of $H_1(D^2, \al^+)$.  Then $c(K') = \gamma' \wedge b_2 \wedge \ldots \wedge b_k$; and $\iota_{\gamma}c(K') \neq 0$.
\end{proof}

\subsection{A relation to Giroux's criterion}

Fix a value of $N$.  Let $f_1$ be an orientation-preserving embedding of $(D^2, \mb{F}(N))$ into $S^2$.  Then, let $f_2$ be an orientation-\emph{reversing} embedding of $(D^2, \mb{F}(N))$ into $S^2$ which agrees with $f_1$ on the boundary, so that the images cover $S^2$.

\begin{defi}
Given two dividing sets $K_1, K_2$ on $(D^2, \mb{F}(N))$, consider the multi-curve $f_1(K_1) \cup_{f_1(\al^+)} f_2(K_2)$ on $S^2$, composed of some number of closed curves.  We say that $K_1$ and $K_2$ are \emph{matchable} if this multi-curve is in fact connected.
\end{defi}

This definition is motivated by Giroux's criterion for $S^2$ \cite{Giroux}, which states that a dividing set on $S^2$ is the dividing set of a convex sphere in a contact 3-manifold with a tight neighborhood if and only if the dividing set is connected.    

\begin{remark}
In \cite{Mathews}, a similar but slightly more precise notion is referred to as \emph{stackability}.  The difference is largely one of perspective: stackability is defined in reference to contact structures; we think of matchability as a combinatorial condition, which we will directly relate to Giroux's criterion below. 
\end{remark}

\begin{theorem}
\label{thm:match}
Let $\Omega^+$ be the unique generator of $V^{top}(D^2, \mb{F}(N)).$
Two dividing sets $K_1, K_2$ on $(D^2, \mb{F}(N))$ are matchable if and only if $c(K_1) \wedge c(K_2) = \Omega^+$. 
\end{theorem}

\begin{proof}
Of course, we may assume that $K_1$ and $K_2$ are non-isolating, and we may then think of $H_1(R^+(K_1), \al^+)$ and $H_1(R^+(K_2), \al^+)$ as subgroups $H_1(D^2, \al^+)$.  The equation holding is equivalent to $H_1(D^2, \al^+)$ being the direct sum of these subgroups.

The dividing sets $K_1$ and $K_2$ will be matchable if and only if the surface $S^+ = f_1\paren{R^+(K_1)} \cup_{f_1(A^+)} f_2\paren{R^+(K_2)}$ is connected and has Euler characteristic 1.  Since $S^+$ is formed by gluing the images of the two regions along $N$ disjoint arcs, we have $\chi\paren{S^+} = \chi\paren{R^+(K_1)} 
+ \chi\paren{R^+(K_2)} - N$.  For $i = 1,2$, $\chi\paren{R^+(K_i)} = N - \mbox{rank}\paren{H_1(R^+(K_i), \al^+)}$; since the rank of $H_1(D^2, \al^+)$ is $N-1$, the Euler characteristic condition is met if and only if 
$$ \mbox{rank}\paren{H_1(R^+(K_1), \al^+)} + \mbox{rank}\paren{H_1(R^+(K_1), \al^+)} = \mbox{rank}\paren{H_1(D^2, \al^+)}.$$

Every component of $S^+$ contains the image of some positive vertex, and the image of every positive vertex lies in some component of $S^+$.  Thus, $S^+$ is connected if and only if every two positive vertices have images that are connected by a path within $S^+$.  We claim that this is equivalent to 
$$ H_1(R^+(K_1), \al^+) + H_1(R^+(K_2), \al^+) = H_1(D^2, \al^+).$$

Suppose that $S^+$ is connected.  Choose two positive vertices $\al_i$ and $\al_j$, and let $\ell$ be such a path in $S^+$ connecting $f_1(\al_i)$ and $f_1(\al_j)$.  After an isotopy, we can decompose $\ell$ as the union of oriented paths $\ell_1, \ldots, \ell_m$ and $\ell'_1, \ldots, \ell'_n$, such that the $\ell_i$ all lie entirely within $f_1(R^+(K_1))$ and have endpoints in $f_1(\al^+)$, and likewise the $\ell'_i$ all lie entirely within $f_2(R^+(K_2))$ and have endpoints in $f_1(\al^+)$.  Then the sum of the classes of the preimages of the $\ell_i$ and $\ell'_i$ in $H_1(D^2, \al^+)$ will be equal to the class of a path running directly from $\al_i$ to $\al_j$.  So every class in $H_1(D^2, \al^+)$ can be written as the sum of such paths, meaning that $H_1(D^2, \al^+)$ is equal to the sum of the two subgroups.

On the other hand, suppose that $H_1(D^2, \al^+)$ is equal to the sum of the two subgroups.  Take a connected path between any two vertices $\al_i$ and $\al_j$, and write it as the sum of a class in $H_1(R^+(K_1), \al^+)$ and a class in $H_1(R^+(K_2), \al^+)$.  Each of these classes may be a disjoint union of paths between different points of $\al^+$.  Considering the boundary map from $H_1(D^2, \al^+)$ to $H_0(\al^+)$, however, there must be a collection of connected paths $\ell_1, \ldots, \ell_m$ such the head of each path is the tail of the next, with $\ell_1$ starting at $\al_i$ and $\ell_m$ ending at $\al_j$, such that each $\ell_k$ lies within $R^+(K_1)$ or within $R^+(K_2)$.  Using these paths, we can easily concoct a curve in $S^+$ connecting $f_1(\al_i)$ and $f_1(\al_j)$. 
 
So connectedness is equivalent to $H_1(D^2, \al^+)$ being the sum of the subspaces, and this sum is a direct sum if and only if the Euler characteristic condition also holds.  Thus $c(K_1) \wedge c(K_2) = \Omega^+$ is equivalent to matchability.
\end{proof}

\subsection{Negative contact elements}
\label{sec:negative}

We wish to show that if we try to define contact elements using $(R^-(K), \al^-)$ in place of $(R^+(K), \al^+)$, we get an invariant that contains essentially equivalent information.  This result seems interesting in its own right, but we will also need it for our application to contact structures on solid tori.

We define the negative contact element in an exactly similar manner to the positive contact element, other than thinking of it as living in the dual of the previously defined sutured surface group.  So, given a dividing set $K$ on $(\Sigma, \mb{F})$, let
$$i^-_K: \Lambda\paren{H_1(R^-(K), \al^-;\Z_2)} \rightarrow \Lambda\paren{H_1(\Sigma, \al^-;\Z_2)} \cong V(\Sigma, \mb{F})^*$$
be the map induced by inclusion; let $\omega^-$ be the unique generator of $\Lambda^{top}\paren{H_1(R^-(K), \al^-)}$; and let $L^-(K) = n(\mb{F}) - \chi\paren{R^-(K)}$.

\begin{defi}
\label{defi:neg}
The \emph{negative contact element} of a dividing set $K$ on a sutured surface $(\Sigma, \mb{F})$ is given by 
$$c^-(K) = \pi_{L^-(K)} \circ i^-_K(\omega^-) \in V(\Sigma, \mb{F})^*.$$
\end{defi} 

We can of course extend this to $\Z$ by choosing homology orientations.

We will refer to the contact element $c(K) \in V(\Sigma, \mb{F})$ that we have been using throughout as the \emph{positive} contact element, and will write it as $c^+(K)$ for the sake of clarity.

\begin{lemma}
\label{lem:ranks}
If $K$ is non-isolating, then
$$ \mbox{rank}\paren{H_1(R^+(K), \al^+)} + \mbox{rank}\paren{H_1(R^-(K), \al^-)} = \mbox{rank}\paren{H_1(\Sigma, \al^+)}.$$
\end{lemma}

\begin{proof}
We again examine part of the long exact sequence of the triple $(\Sigma, R^+(K), \al^+)$,
$$  H_1(R^+(K), \al^+) \rightarrow H_1(\Sigma, \al^+) \rightarrow H_1\paren{\Sigma, R^+(K)}  \rightarrow H_0(R^+(K), \al^+). $$
Since $K$ is non-isolating, the first map is injective (by Proposition \ref{prop:isocomps}), and the last group is trivial.  
Using excision and Poincar\'{e} duality,
$$ H_1\paren{\Sigma, R^+(K)} \cong H_1(R^-(K), K) \cong H^1(R^-(K), \al^-).$$
The conclusion follows.
\end{proof}

To relate the negative and positive formulations, let $\Omega^+$ be the unique generator of $V^{top}(\Sigma, \mb{F})$; likewise, let $\Omega^-$ be the generator of $\Lambda^{top}\paren{H_1(\Sigma, \al^-; \Z_2)} = V^{top}(\Sigma, \mb{F})^*$.  These elements satisfy the equation
$$\langle \Omega^- | \Omega^+ \rangle = 1.$$ 

It is easy to see that the map taking $x \in V(\Sigma, \mb{F})$ to $\iota_x \Omega^- \in V(\Sigma, \mb{F})^*$ is a linear isomorphism, as is the map taking $x \in V(\Sigma, \mb{F})^*$ to $\iota_x \Omega^+ \in \paren{V(\Sigma, \mb{F})^*}^* = V(\Sigma, \mb{F})$. The next Proposition asserts that these isomorphisms trade positive and negative contact elements.  

\begin{prop}
\label{prop:neg}
$c^+(K) = \iota_{c^-(K)}\Omega^+$ and $c^-(K) = \iota_{c^+(K)}\Omega^-$.
\end{prop}

\begin{proof}
We prove the first result; the second is similar.  Of course, both results are clear if $K$ is isolating by Proposition \ref{prop:isocomps} (which clearly applies equally replacing $+$ with $-$), so we assume going forward that $K$ is non-isolating.  In particular, we may view $H_1(R^+(K), \al^+)$ and $H_1(R^-(K), \al^-)$ as subspaces of $H_1(\Sigma, \al^+)$ and $H_1(\Sigma, \al^-)$, respectively.  
Letting $n_+$ and $n_-$ be the ranks of $H_1(R^+(K), \al^+)$ and $H_1(R^-(K), \al^-)$, respectively, Lemma \ref{lem:ranks} then says that the rank of $H_1(\Sigma, \al^+)$ is $n_+ + n_-$, which we will denote by $n$.

Since $K$ is non-isolating, we may choose a basis $\{b_1, \ldots, b_{n}\}$ of $H_1(\Sigma, \al^+)$, with dual basis $\{c_1, \ldots, c_n\}$ of $H_1(\Sigma, \al^-)$, such that $c^-(K) = c_{[1, n_+]}$ (recall notation from Section \ref{sec:extnot}, and note that we don't rule out $c_{\emptyset} = 1$).  Of course, $b_{[1, n]} = \Omega^+$ and $c_{[1,n]} = \Omega^-$, and the sets of elements $\{b_I | I \subset [1,n]\}$ and $\{c_I | I \subset [1,n]\}$ form dual bases of $V(\Sigma, \mb{F})$ and $V(\Sigma, \mb{F})^*$ respectively.  
For a subset $I$ of $[1,n]$, 
$$ \langle c_I | \iota_{c^-(K)}\Omega^+ \rangle = \langle c^-(K) \wedge c_I | \Omega^+ \rangle = \langle c_{[1, n_+]} \wedge c_I | \Omega^+ \rangle = \left\{\begin{array}{ll}
1, & I = [n_+ + 1, n]\\
0, & \mbox{otherwise}.
\end{array} \right.$$ 
Thus, $\iota_{c^-(K)}\Omega^+ = b_{[n_+ + 1, n]}$, which we must show is equal to $c^+(K)$.  

Equivalently, we may say that $H_1(R^-(K),\al^-)$ is the span of $\{c_1, \ldots, c_{n_+}\}$, and that we must show that $H_1(R^+(K), \al^+)$ is spanned by $\{b_{n_+ + 1}, \ldots, b_n\}$ -- that is, that $H_1(R^+(K), \al^+)$ is exactly the subgroup of $H_1(\Sigma, \al^+)$ annihilated by $H_1(R^-(K), \al^-)$.  Of course, any element of $H_1(R^+(K), \al^+)$ can be isotoped not to intersect any element of $H_1(R^-(K), \al^-)$.  Furthermore, the rank of the subgroup of $H_1(\Sigma, \al^+)$ annihilated by $H_1(R^-(K), \al^-)$ will be $n - n_- = n_+$.  Thus, $H_1(R^+(K), \al^+)$ is exactly the annihilated subgroup, as needed.
\end{proof}

\subsection{Contact structures on solid tori}

Let $(T, \xi)$ be a contact solid torus with convex boundary.  Let $\mu$ be a curve on the boundary which bounds a meridional disk, and $\lambda$ a longitude of the boundary, with the two oriented so that $\mu \cdot \lambda = 1$.  If we assume that there is a neighborhood of $\partial T$ that is tight, the dividing set $\Gamma \subset \partial T$ will be $2n$ parallel curves, each representing some class $p\mu + q\lambda$ for relatively prime $p,q$ and $q > 0$.   

Let $D \subset T$ be a properly embedded meridional disk with $\partial D \subset \partial T$; by perhaps performing a small isotopy, arrange that $\partial D$ intersects $\Gamma$ tranversally, and that $D$ is a convex surface with Legendrian boundary.  Take a tubular neighborhood $D \times [0,1]$ of $D$, where the vector field in the $[0,1]$ direction is a contact vector field; denote by $B$ the space obtained by removing $D \times (0,1)$ from the torus (i.e. $B$ is just $T$ cut along $D$).  Topologically, $B$ will be a 3-ball; the boundary of $B$ is composed of $D \times \{0\}$, $-(D \times \{1\})$, and the annulus $\partial T \setminus (\partial D \times (0,1))$.  

 Consider the dividing set $K$ on $D$, as well as the dividing set $\Gamma'$ on the annulus gotten by cutting open $\Gamma$.  Note that $K \cap \partial D$ and $\Gamma' \cap \partial D$ alternate as one goes around $\partial D$; see for example \cite{Honda}. By the edge-rounding lemma of \cite{Honda}, we may round the edges joining the three components of $\partial B$ to form a smooth, convex $S^2$ boundary, and the dividing sets on the three components get modified as in Figure \ref{fig:edgeround}: the dividing curves are smoothed by having them turn towards the right as they pass from one component to another.

We may identify $D$ with the sutured surface $\paren{D^2, \mb{F}(nq)}$, so that the set of sutures $F^{\pm}$ is $K \cap \partial D$, and the set of vertices $\al^{\pm}$ is $\Gamma' \cap \partial D$.  Fix such an identification with $\paren{D^2, \mb{F}(nq)}$.  Recall that we number the vertices we encounter going counterclockwise along $\partial D^2$ by $\al_1, \al_2, \ldots, \al_{2nq}$, where $\al_1$ is a positive vertex.  For $i \in \Z$, let $\beta_i$ be an arc in $D^2$ with $\partial \beta_i = \al_{i+2} - \al_i$, where $\al_i = \al_{j}$ if $i \equiv j \,(\mbox{mod} 2nq)$. As we have already mentioned, $\{\beta_1, \beta_3, \ldots, \beta_{2nq - 3}\}$ will be a basis for $H_1(D^2, \al^+)$; on the other hand, $\beta_i$ for $i$ even will represent an element of $H_1(D^2, \al^-) = H^1(D^2, \al^+)$ instead of $H_1(D^2, \al^+)$.  

For an \emph{odd} number $j$, let $\varphi_j: H_1(D^2, \al^+) \rightarrow H^1(D^2, \al^+)$ be the map which sends $\beta_i$ to $\beta_{i+j}$ -- we emphasize that this turns homology elements into cohomology elements.

\begin{figure}[htbp]
	\centering
		\includegraphics[width=0.6\textwidth]{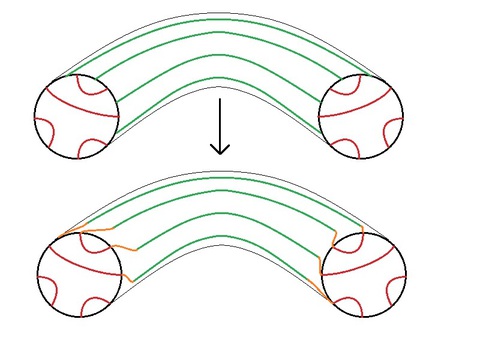}
	\caption{The top figure shows the solid torus $T$ cut open along a neighborhood of $D$.  The left disk represents $D \times \{0\}$, while the right disk is $-(D \times \{1\})$.  The green curves are $\Gamma$, the red curves are the images of the dividing set $K$ on $D$.  The bottom figure shows the edge-rounded dividing set on $\partial B$, as prescribed by the lemma of \cite{Honda}.  Note that we have pushed all the alterations to the dividing sets onto $\partial T \setminus \partial D$.}
	\label{fig:edgeround}
\end{figure}

\begin{theorem}
The restriction of $\xi$ to a neighborhood of $\partial B$ is tight if and only if the dividing set $K$ on $D$ satisfies
$$ \langle \varphi_{2np+1}\paren{c(K)} | c(K) \rangle = 1.$$
\end{theorem} 

\begin{proof}
The restriction is tight if the rounded dividing set on $\partial B$ is connected.  We would like to analyze this condition using Theorem \ref{thm:match}.  
So, let $f_1$ be the previously fixed orientation-preserving embedding of $\paren{D^2, \mb{F}(nq)}$ onto $D\times \{0\}$, and let $f_2$ be an orientation-\emph{reversing} embedding of $\paren{D^2, \mb{F}(nq)}$ onto $\paren{\partial T \setminus (\partial D \times (0,1))} \cup -(D \times \{1\})$, as in the definition of matchable dividing sets, so that the images of the two embeddings cover $\partial B$.  Then, choose a dividing set $K'$ on $\paren{D^2, \mb{F}(nq)}$ such that the smoothed dividing set on $\partial B$ is gotten by gluing the image of $K$ under $f_1$ to the image of $K'$ under $f_2$.  So, $K'$ is just $K$ rotated, including the small twists required by the ``turn right'' edge-rouding condition.  In fact, it is not hard to see that if the $2nq$ sutures are evenly spaced on the unit disk in $\mathbb{C}$, then $K'$ is gotten by rotating counterclockwise by an angle of $\frac{2np+1}{2nq}\cdot 2\pi$.

Therefore, $\varphi_{2np+1}\paren{c(K)}$ will be the \emph{negative} contact element for $K'$, as the odd rotation of a basis of $H_1(R^+(K), \al^+)$ will yield a basis of $H_1(R^-(K'), \al^-)$.   The arguments given in Theorem \ref{thm:match} clearly apply equally to negative contact elements, and thus $K$ and $K'$ are matchable if and only if 
$$c^-(K) \wedge c^-(K')  = \Omega^-,$$
  or equivalently, if
$$ \langle c^-(K) \wedge\varphi_{2np+1}\paren{c(K)}| \Omega^+ \rangle = 1.$$
But by the definition of the interior product and Proposition \ref{prop:neg}, the left side is equal to
$$ \langle \varphi_{2np+1}\paren{c(K)}| \iota_{c^-(K)}\Omega^+ \rangle = \langle \varphi_{2np+1}\paren{c(K)}| c(K) \rangle = 1.$$
\end{proof}


\begin{thebibliography}{100}

\bibitem{Gabai} D. Gabai, \emph{Foliations and the topology of 3-manifolds.} J. Differential Geom. 18 (1983), 445–-503.

\bibitem{Giroux} E. Giroux, \emph{Convexit\'{e} en topologie de contact.} Comment. Math. Helv. 66 (1991), no. 4, 637-–677.

\bibitem{contcat} K. Honda, \emph{Contact structures, Heegaard Floer homology and triangulated categories.} In preparation.

\bibitem{Honda} K. Honda, \emph{On the classification of tight contact structures I.} Geom. Topol. 4 (2000), 309-–368

\bibitem{HondaII} K. Honda, \emph{On the classification of tight contact structures II.} J. Differential Geom. 55 (2000), 83--143

\bibitem{HKM} K. Honda, W. Kazez, G. Mati\'{c}, \emph{Contact Structures, Sutured Floer Homology and TQFT.} arXiv:0807.2431, 2008.

\bibitem{HKM1}  K. Honda, W. Kazez, G. Mati\'{c}, \emph{The contact invariant in sutured Floer homology.}  Invent. Math. 176 (2009), no. 3, 637-676.

\bibitem{juhasz} A. Juh\'{a}sz, \emph{Holomorphic discs and sutured manifolds.} Algebr. Geom. Topol. 6 (2006), 1429–-1457.

\bibitem{Massot} P. Massot, \emph{Infinitely many universally tight torsion free contact structures with vanishing Ozsv\'{a}th-Szab\'{o} invariants.} Mathematische Annalen 353 (2012), no. 4, 1351 -- 1376.

\bibitem{Mathews} D. Mathews, \emph{Chord diagrams, contact-topological quantum field theory, and contact categories.}  Algebr. Geom. Topol. 10 (2010), no. 4, 2091 -- 2189.

\bibitem{Mathews2} D. Mathews, \emph{Sutured 
oer homology, sutured TQFT and non-commutative QFT.} Algebr. Geom. Topol. 11 (2011), no. 5, 2681 -- 2739.

\bibitem{Mathews3} D. Mathews, \emph{Sutured TQFT, torsion, and tori.} arXiv:1102.3450, 2011.

\bibitem{itsybitsy} D. Mathews, \emph{Itsy bitsy topological field theory.} arXiv:1201.4584, 2012.

\bibitem{Ni} Y. Ni, \emph{Homological actions on sutured Floer homology.} arXiv:1010.2808, 2010.

\bibitem{ozsz} P. Ozsv\'{a}th, Z. Sz\'{a}bo, \emph{Holomorphic disks and topological invariants for closed three-manifolds.} Ann. of
Math. (2) 159 (2004), 1027–-1158.
\end{thebibliography}
\end{document}